\documentclass[11pt, fleqn]{amsart}
\usepackage{amscd, amsfonts, amsthm, graphicx, amssymb}
\usepackage[all]{xy}

\newtheorem{theorem}{Theorem}[section]
\newtheorem{proposition}{Proposition}[section] 
\newtheorem{corollary}{Corollary}[section] 
\newtheorem{lemma}{Lemma}[section]

\theoremstyle{definition}

\newtheorem{defn}{Definition}[section]
\newtheorem{remark}{Remark}[section]
\newtheorem{notation}{Notation}[section] 
 
\newtheorem{example}{Example}[section]

\newtheorem*{compl}{Complement to Theorem 2.1}

\newtheorem{conjecture}{Conjecture}[section]

\newcommand{\ds}{\displaystyle}

\newcommand{\lmt}{\longmapsto}
\newcommand{\cal}{\mathcal}
\renewcommand{\phi}{\varphi}

\newcommand{\lra}{\longrightarrow}

\newcommand{\diam}{\operatorname{diam}}

\newcommand{\e}{\varepsilon}
\newcommand{\id}{\mathrm{id}}

\newcommand{\E}{\mathcal E}
\newcommand{\F}{\mathcal F}

\newcommand{\U}{\mathcal U}

\newcommand{\W}{\mathcal W}

\newcommand{\IR}{\mathbb R}
\newcommand{\IN}{\mathbb N}
\newcommand{\IZ}{\mathbb Z}

\begin{document}
\baselineskip 6 mm

\thispagestyle{empty}

%%%%%%%%%%%%%%%%%%%

\title[The groups of uniform homeomorphisms]{Groups of uniform homeomorphisms of covering spaces} 

\author[Tatsuhiko Yagasaki]{Tatsuhiko Yagasaki}
\address{Graduate School of Science and Technology, Kyoto Institute of Technology, Kyoto, 606-8585, Japan}
\email{yagasaki@kit.ac.jp}

\subjclass[2010]{57S05; 58D10, 57N15, 54E40} 
\keywords{Space of uniform embeddings, Group of uniform homeomorphisms, Uniform topology, Euclidean ends}

\maketitle

\begin{abstract} 
In this paper we deduce a local deformation lemma for uniform embeddings in a metric covering space over a compact manifold 
from the deformation lemma for embeddings of a compact subspace in a manifold. 
This implies the local contractibility of the group of uniform homeomorphisms of such a metric covering space under the uniform topology. 
Furthermore, combining with similarity transformations, this enables us to induce  
a global deformation property of groups of uniform homeomorphisms of metric spaces with Euclidean ends.
In particular, we show that the identity component of the group of uniform homeomorphisms of the standard Euclidean $n$-space is contractible. 
\end{abstract} 

\section{Introduction}

In this paper we study some local and global deformation properties of spaces of uniform embeddings and 
groups of uniform homeomorphisms of metric covering spaces over compact manifolds and metric spaces with Euclidean ends. 

Suppose $(X,d)$ and $(Y, \rho)$ are metric spaces. 
A map $h : (X,d) \to (Y, \rho)$ is said to be uniformly continuous if for each $\e > 0$ there is a $\delta > 0$ such that 
if $x,x' \in X$ and $d(x,x') < \delta$ then $\rho(f(x), f(x')) < \e$. 
The map $h$ is called a uniform homeomorphism if $h$ is bijective and both $h$ and $h^{-1}$ are uniformly continuous. 
A uniform embedding is a uniform homeomorphism onto its image. 

In \cite{EK}  R.\,D.~Edwards and R.\,C.~Kirby obtained a fundamental local deformation theorem for embeddings of a compact subspace in a manifold. 
Based upon this theorem, in this article we deduce a local deformation lemma for uniform embeddings in a metric covering space over a compact manifold. 
Here, the Arzela-Ascoli theorem (\cite[Theorem 6.4]{Du}) plays an essential role in order to pass from the compact case to the uniform case. 

Suppose $(M, d)$ is a topological manifold possibly with boundary with a fixed metric $d$ and $X$, $C$ are subspaces of $M$. 
Let $\E^u_\ast(X, M; C)$ denote the space of uniform proper embeddings $f : (X, d|_X) \to (M, d)$ such that $f = \id$ on $X \cap C$. 
This space is endowed with the uniform topology induced from the  sup-metric 
$$d(f,g) = \sup \big\{ d(f(x), g(x)) \mid x \in X \big\} \in [0, \infty] \hspace{8mm} (f, g \in \E^u_\ast(X, M; C)).$$ 

Since the notion of uniform continuity depends on the choice of metric $d$ on the manifold $M$, 
it is necessary to select a reasonable class of metrics. 
In \cite{Ce} (cf, \cite[Section 5.6]{Ru}) A.V.~{\v C}ernavski\u\i \break considered the case where $M$ is the interior of a compact manifold $N$ and the metric $d$ is a restriction of some metric on $N$. 
In this article we consider the case where $M$ is a covering space over a compact manifold $N$ and the metric $d$ is the pull-back of some metric on $N$. 
The natural model is the class of Riemannian coverings in the smooth category. 
In order to remove the extra requirements in the smooth setting, here we introduce the notion of metric covering projection. 
Its definition and basic properties are included in Section 2.2 below.  
The following is our main theorem. 

\begin{theorem}\label{thm_local_deformation} 
Suppose $\pi : (M, d) \to (N, \rho)$ is a metric covering projection, $N$ is a compact topological $n$-manifold possibly with boundary,  
$X$ is a closed subset of $M$, $W' \subset W$ are uniform neighborhoods of $X$ in $(M, d)$ and 
$Z$, $Y$ are closed subsets of $M$ such that $Y$ is a uniform neighborhood of $Z$. 
Then there exists a neighborhood $\W$ of the inclusion map $i_W : W \subset M$ in $\E^u_\ast(W, M; Y)$ and 
a homotopy $\phi : \W \times [0,1] \lra \E^u_\ast(W, M; Z)$ such that 
\begin{itemize} 
\item[(1)] for each $h \in \W$ \\ 
\begin{tabular}[t]{c@{\ \,}l}
{\rm (i)} & $\phi_0(h) = h$, \hspace{3mm} 
{\rm (ii)} $\phi_1(h) = \id$ \ on \ $X$, \\[2mm] 
{\rm (iii)} & $\phi_t(h) = h$ \ on \ $W - W'$ \ \ and \ \ $\phi_t(h)(W) = h(W)$ \ \ $(t \in [0,1])$, \\[2mm] 
{\rm (iv)} & if $h = \id$ on $W \cap \partial M$, then $\phi_t(h) = \id$ on $W \cap \partial M$ $(t \in [0,1])$, 
\end{tabular} 
\vskip 1.5mm 
\item[(2)] $\phi_t(i_W) = i_W$ \ $(t \in [0,1])$.
\end{itemize} 
\end{theorem}

This theorem induces some consequences on the theory of uniform homeomorphisms. 
Suppose $(X,d)$ is a metric space and $A$ is a subset of $X$. 
Let ${\cal H}^u_A(X,d)$ denote the group of uniform homeomorphisms of $(X, d)$ onto itself which fix $A$ pointwise, 
endowed with the uniform topology. 
Let ${\cal H}^u_A(X, d)_0$ denote the connected component of the identity map $\id_X$ of $X$ in ${\cal H}_A^u(X, d)$. 
We are also concerned with the subgroup 
$${\cal H}^u_A(X, d)_b = \{ h \in {\cal H}_A^u(X, d) \mid d(h, \id_X) < \infty \}.$$ 
It is easily seen that ${\cal H}^u_A(X, d)_0 \subset {\cal H}^u_A(X, d)_b$ since ${\cal H}^u_A(X, d)_b$ is both closed and open in ${\cal H}^u_A(X, d)$. 
When $X - A$ is relatively compact in $X$, the group  
${\cal H}^u_A(X,d)$ coincides with the whole group of homeomorphisms of $X$ onto itself which fix $A$ pointwise endowed with the compact-open topology. 
In this case we delete the script ``$u$'' from the notation. As usual, the symbol $A$ is suppressed when it is an empty set. 

In \cite{Ce} it is shown that ${\cal H}^u(M, d)$ is locally contractible in the case where  
$M$ is the interior of a compact manifold $N$ and the metric $d$ is a restriction of some metric on $N$.  
The next corollary is a direct consequence of Theorem~\ref{thm_local_deformation}. 

\begin{corollary}\label{cor_local_contractibility} 
Suppose $\pi : (M, d) \to (N, \rho)$ is a metric covering projection onto a compact topological $n$-manifold $N$ possibly with boundary.
Then ${\cal H}^u(M, d)$ is locally contractible.  
\end{corollary}

Next we study a global deformation property of the group ${\cal H}^u(X,d)$. 
The most standard example is the $n$-dimensional Euclidean space $\IR^n$ with the standard Euclidean metric. 
The relevant feature in this scenario is the existence of similarity transformations. 
This enables us to deduce a global deformation of uniform embeddings from a local one. 

To be more general, 
we treat metric spaces with bi-Lipschitz Euclidean ends. 
Recall that a map $h : (X,d) \to (Y, \rho)$ between metric spaces is said to be Lipschitz if there exists a constant $C > 0$ 
such that $\rho(h(x), h(x')) \leq Cd(x, x')$ for any $x, x' \in X$. 
The map $h$ is called a bi-Lipschitz homeomorphism if $h$ is bijective and both $h$ and $h^{-1}$ are Lipschitz maps. 
The model of Euclidean end is the complement $\IR^n_r = \IR^n - O(r)$ of the round open $r$-ball $O(r)$ centered at the origin. 
These complements $\IR^n_r$ $(r > 0)$ are bi-Lipschitz homeomorphic to each other under  similarity transformations. 
A bi-Lipschitz $n$-dimensional Euclidean end of a metric space $(X, d)$ means a closed subset $L$ of $X$ 
which admits a bi-Lipschitz homeomorphism of pairs, 
$\theta : (\IR^n_1, \partial \IR^n_1) \approx ((L, {\rm Fr}_X L), d|_L)$ and $d(X - L, L_r) \to \infty$ as $r \to \infty$, where 
${\rm Fr}_X L$ is the topological frontier of $L$ in $X$ and 
$L_r = \theta(\IR^n_r)$ for $r \geq 0$. 
We set $L' = \theta(\IR^n_2)$ and $L'' = \theta(\IR^n_3)$. 
Using similarity transformations, we can deduce the following result from the local deformation theorem, Theorem~\ref{thm_local_deformation}. 

\begin{theorem}\label{thm_Euclid-end}
Suppose $X$ is a metric space and $L_1, \cdots, L_m$ are mutually disjoint bi-Lipschitz Euclidean ends of $X$. 
Let $L' = L_1' \cup \cdots \cup L_m'$ and $L'' = L_1'' \cup \cdots \cup L_m''$. 
Then there exists a strong deformation retraction $\phi$ of ${\cal H}^u(X)_b$ onto ${\cal H}^u_{L''}(X)$ such that 
$$\mbox{$\phi_t(h) = h$ \ on \ $h^{-1}(X - L') - L'$ \ \ for any \ $(h,t ) \in {\cal H}^u(X)_b \times [0,1]$.}$$  
\end{theorem} 

\begin{example}\label{example}
(1) ${\cal H}^u(\IR^n)_b$ is contractible for every $n \geq 0$. 
In fact, $\IR^n$ has the model Euclidean end $\IR^n_1$ and hence 
there exists a strong deformation retraction of ${\cal H}^u(\IR^n)_b$ onto 
${\cal H}^u_{\IR^n_3}(\IR^n)$. 
The latter is contractible by Alexander's trick. 

(2) The $n$-dimensional cylinder $M = {\Bbb S}^{n-1} \times \IR$ is the product of the $(n-1)$-sphere ${\Bbb S}^{n-1}$  and the real line $\IR$. If $M$ is asigned a metric so that ${\Bbb S}^{n-1} \times (-\infty, -1]$ and 
${\Bbb S}^{n-1} \times [1, \infty)$ are two bi-Lipschitz Euclidean ends of $M$, then 
${\cal H}^u(M)_b$ includes the subgroup 
${\cal H}_{{\Bbb S}^{n-1} \times \IR_1}(M) \approx {\cal H}_\partial({\Bbb S}^{n-1} \times [-1,1])$ as a strong deformation retract. 
In particular, ${\cal H}^u(M)_0$ admits a strong deformation retraction onto ${\cal H}_{{\Bbb S}^{n-1} \times \IR_1}(M)_0 \approx {\cal H}_\partial({\Bbb S}^{n-1} \times [-1,1])_0$. 

(3) In dimension 2, we have a more explicit conclusion. Suppose $N$ is a compact connected 2-manifold with a nonempty boundary and $C = \cup_{i=1}^m C_i$ is a nonempty union of some boundary circles of $N$. 
If the noncompact 2-manifold $M = N - C$ is assigned a metic $d$ such that for each $i = 1, \cdots, m$ the end $L_i$ of $M$ corresponding to the boundary circle $C_i$ is a bi-Lipschitz Euclidean end of $(M, d)$, then ${\cal H}^u(M, d)_0 \simeq {\cal H}^u_{L''}(M)_0 \approx {\cal H}_{C}(N)_0 \simeq \ast$. 
\end{example}

\begin{remark} In Example~\ref{example} (1), 
one might expect that conjugation by a suitable shrinking homeomorphism $\IR^n \approx O(1)$ and 
extension by the identity on the boundary would directly reduce the problem to the case of 
${\cal H}_\partial(B(1))$, the group of homeomorphisms of the closed unit ball relative to the boundary, 
since this group is contractible by Alexander's trick. 
However, the contraction of ${\cal H}^u(\IR^n)_b$ obtained in this way is not continuous. 
In fact, 
%Otherwise, 
it would mean that any $h \in {\cal H}^u(\IR^n)_b$ could be approximated by compactly supported homeomorphisms in the sup-metric. 
But this does not hold, for example, for any translation $h(x) = x +a$ $(a \neq 0)$. 
\end{remark} 

In \cite{MSYY} we studied the topological type of ${\cal H}^u(\IR)_b$ as an infinite-dimensional manifold 
and showed that it is homeomorphic to $\ell_\infty$. 
Example 1.1 leads to the following conjecture.  

\begin{conjecture} ${\cal H}^u(\IR^n)_b$ is homeomorphic to $\ell_\infty$ for any $n \geq 1$.  
\end{conjecture} 

This paper is organized as follows. Section 2 includes some preliminary results on 
metric covering projections and spaces of uniform embeddings. 
Section 3 is devoted to the proof of Theorem~\ref{thm_local_deformation} and 
the final section, Section 4, includes the proof of Theorem~\ref{thm_Euclid-end}. 

\section{Preliminaries}
\subsection{Conventions} \mbox{} 

%The symbols $\IN$ and $\IZ_{\geq 0}$ denote the sets of positive integers and nonnegative integers respectively. 
%For $m \in \IZ_{\geq 0}$ we use the notations $[m] = \{0, 1, \cdots, m \}$ and $[m]_+ = \{1, \cdots, m \}$.

Maps between topological spaces are assumed to be continuous. 
The word ``function'' means a correspondence not assumed to be continuous. 
For a topological space $X$ and a subset $A$ of $X$, 
the symbols ${\rm Int}_X A$,  $cl_X A$ and ${\rm Fr}_X A$ denote the topological 
interior, closure and frontier of $A$ in $X$. 
The identity map on $X$ is denoted by $\id_X$, while the inclusion map $A \subset X$ is denoted by $i_A$, $\iota_A$ or $\id_A$, etc. 
When ${\cal F}$ is a collection of subsets of $X$, 
the union of ${\cal F}$ is denoted by $|\F|$ or $\bigcup \F$.
For $A \subset X$ 
the star of $A$ with respect to $\cal{F}$ is defined by ${\rm St}(A, \cal{F}) = A \cup \big( \cup \{ F \in {\cal F} \mid F \cap A \neq  \emptyset \}\big) \subset X$. 
%Let ${\rm st}\,\F = \{ {\rm st}\,(F, \F) \mid F \in \F \}$. 

For an $n$-manifold $M$, 
the symbols $\partial M$ and ${\rm Int}\,M$ denote the boundary and interior of $M$ as a manifold. 

\subsection{Metric covering projections} \mbox{} 

Suppose $(X, d)$ is a metric space. 
(Below, when the metric $d$ is implicitly understood, we eliminate the symbol $d$ from the notations.)  
The distance between two subsets $A, B$ of $X$ is defined by  
$d(A, B) = \inf \{ d(x, y) \mid x \in A, y \in B \}$. 
For $\delta \geq 0$ 
the closed $\delta$-neighborhood of $A$ in $X$ is defined by 
$C_\delta(A) = \{ x \in X \mid d(x, A) \leq \delta \}$. 
%Note that $C_0(A) = cl_X A$. 
%For a collection $\F$ of subsets of $X$, 
%we set $\F_\delta = \{ C_\delta(F) \mid F \in \F \}$. 

A neighborhood $U$ of $A$ in $X$ is called a uniform neighborhood of $A$ in $(X, d)$ 
if $C_\delta(A) \subset U$ for some $\delta >0$. 
For $\e > 0$ a subset $A$ of $X$ is said to be $\e$-discrete if 
$d(x,y) \geq \e$ for any distinct points $x, y \in A$. 
More generally, a collection ${\cal F}$ of subsets of $X$ 
is said to be $\e$-discrete if $d(F, F') \geq \e$ for any $F, F' \in \F$ with $F \neq F'$. 
We say that $A$ or $\F$ is uniformly discrete if it is $\e$-discrete for some $\e > 0$. 

For the basics on covering spaces, one can refer to \cite[Chapter 2, Section 1]{Sp}.
If $p : M \to N$ is a covering projection and $N$ is a topological $n$-manifold possibly with boundary, then so is $M$ and $\partial M = \pi^{-1}(\partial N)$. 

\begin{defn} A covering projection $\pi : (X, d) \to (Y, \rho)$ between metric spaces is called a metric covering projection 
if it satisfies the following conditions: 
\begin{itemize} 
\item[$(\natural)_1$] There exists an open cover ${\cal U}$ of $Y$ such that for each $U \in {\cal U}$ the inverse $\pi^{-1}(U)$ is the disjoint union of open subsets of $X$ each of which is mapped isometrically onto $U$ by $\pi$. 
\item[$(\natural)_2$] For each $y \in Y$ the fiber $\pi^{-1}(y)$ is uniformly discrete in $X$. 
\item[$(\natural)_3$] $\rho(\pi(x), \pi(x')) \leq d(x, x')$ for any $x, x' \in X$. 
\end{itemize}
\end{defn}

When an open subset $U$ of $Y$ satisfies the condition in $(\natural)_1$, 
we say that $U$ is isometrically evenly covered by $\pi$. 
In this case, if $U$ is connected, then each connected component of $\pi^{-1}(U)$ is mapped isometrically onto $U$ by $\pi$. 

Riemannian covering projections are typical examples of metric covering projections. 

\begin{lemma}\label{lemma_covering_proj}
Suppose $\pi : (X, d) \to (Y, \rho)$ is a metric covering projection and $Y$ is compact. 
\begin{itemize}
\item[(1)] There exists $\e > 0$ such that each fiber of $\pi$ is $\e$-discrete.
\item[(2)] Suppose $U$ is an open subset of $Y$ and $V$ is an open subset of $\pi^{-1}(U)$ which is mapped isometrically onto $U$ by $\pi$, 
$E$ is a subset of $V$ and $F = \pi(E) \subset U$. 
Then $d(X - V, E) \geq \min \{ \e/2, \rho(Y - U, F) \}$. 
\end{itemize} 
\end{lemma} 

\begin{proof} 
(1) By $(\natural)_1$, $(\natural)_2$ for each $y \in Y$ we can find 
\begin{itemize}
\item[(i)\ ] $\e_y > 0$ such that $\pi^{-1}(y)$ is $3\e_y$-discrete and 
\item[(ii)\,] an open neighborhood $U_y$ of $y$ in $Y$ such that $\diam U_y \leq \e_y$ and $U_y$ is isometrically evenly covered by $\pi$, that is, 
$\pi^{-1}(U_y)$ is the disjoint union of some open subsets $V_y^\lambda$ $(\lambda \in \Lambda_y)$ of $X$ and 
each $V_y^\lambda$ is mapped isometrically onto $U_y$ by $\pi$. 
\end{itemize}
We show that the family $\{ V_y^\lambda \}_{\lambda \in \Lambda_y}$ is $\e_y$-discrete. In particular, for any $z \in U_y$ 
the fiber $\pi^{-1}(z)$ is $\e_y$-discrete. 

To see this claim, take any $\lambda, \mu \in \Lambda_y$ with $\lambda \neq \mu$. 
We have to show that $d(V_y^\lambda, V_y^\mu) \geq \e_y$. 
Let $y_\lambda \in V_y^\lambda$ and $y_\mu \in V_y^\mu$ be the points such that 
$\pi(y_\lambda) = \pi(y_\mu) = y$. 
Then, for any $x_\lambda \in V_y^\lambda$ and $x_\mu \in V_y^\mu$ it follows that 
\begin{itemize}
\item[] $d(x_\lambda, y_\lambda) \leq \diam V_y^\lambda = \diam U_y \leq \e_y$ \ \ and \ \ 
$d(x_\mu, y_\mu) \leq \diam V_y^\mu = \diam U_y \leq \e_y$, \ \ so that 
\vskip 0.5mm 
\item[] $3 \e_y \leq d(y_\lambda, y_\mu) \leq d(y_\lambda, x_\lambda) + d(x_\lambda, x_\mu) + d(x_\mu, y_\mu) \leq d(x_\lambda, x_\mu) + 2 \e_y$ 
\ \ and \ \ $d(x_\lambda, x_\mu) \geq \e_y$. 
\end{itemize}

Since $Y$ is compact, there exist finitely many points $y_1, \cdots, y_n \in Y$ such that $\{ U_{y_1}, \cdots, U_{y_n} \}$ covers $Y$. 
Then $\e = \min \{ \e_{y_1}, \cdots, \e_{y_n} \}$ satisfies the required condition. 

(2) Take any points $x \in E$ and $x' \in X - V$. Let $y = \pi(x)$ and $y' = \pi(x')$. 
\begin{itemize}
\item[(i)\ ] the case that $x' \in \pi^{-1}(U) - V$; Let $x'' \in V$ be the point such that $\pi(x'') = y'$. 
Since $\pi : (V, d) \to (U, \rho)$ is an isometry, we have $d(x, x'') = \rho(y, y')$. 
From $(\natural)_3$ it follows that $\rho(y,y') \leq d(x,x')$.  
Therefore, 
$\e \leq d(x',x'') \leq d(x', x) + d(x, x'') \leq 2d(x', x)$ and $d(x', x) \geq \e/2$. 
\item[(ii)\,] the case that $x' \in X - \pi^{-1}(U)$;  
By $(\natural)_3$ we have $d(x,x') \geq \rho(y,y') \geq \rho(F, Y - U)$.  
\end{itemize}
This implies the assertion. 
\end{proof} 

\subsection{Spaces of uniformly continuous maps} \mbox{} 

First we list some basic facts on the uniform topology on the space of uniformly continuous maps. 
Recall that the definitions of uniformly continuous maps, uniform homeomorphisms and uniform embeddings are included in Section 1.
Below $(X, d)$, $(Y, \rho)$ and $(Z, \eta)$ denote metric spaces. (The metrics $d$, $\rho$ and $\eta$ are also denoted by the symbols $d_X$, $d_Y$ and $d_Z$ respectively. 
As usual, when these metrics are implicitly understood, we eliminate them from the notations.)  
Let ${\cal C}(X, Y)$ and ${\cal C}^u((X, d), (Y, \rho))$ 
denote the space of maps $f : X \to Y$ and 
the subspace of uniformly continuous maps $f : (X, d) \to (Y, \rho)$.  
The metric $\rho$ on $Y$ induces the sup-metric on ${\cal C}(X, Y)$ defined by 
$$\rho(f,g) = \sup \{ \rho(f(x), g(x)) \mid x \in X \} \in [0, \infty].$$ 
The topology on ${\cal C}(X, Y)$ induced by this sup-metric $\rho$ is called the uniform topology.
Below the space ${\cal C}(X,Y)$ and its subspaces are endowed with the sup-metric $\rho$ and 
the uniform topology, otherwise specified. To emphasize this point, sometimes we use the symbol ${\cal C}(X, Y)_u$. 
On the other hand, when the space ${\cal C}(X, Y)$ is endowed with the compact-open topology, 
we use the symbol ${\cal C}(X, Y)_{co}$.  
When $X$ is compact, we have ${\cal C}^u((X, d), (Y, \rho))_u = {\cal C}(X, Y)_{co}$.  

It is important to notice that the composition map 
$${\cal C}^u((X, d), (Y, \rho))_u \times {\cal C}^u((Y, \rho), (Z, \eta))_u \lra {\cal C}^u((X, d), (Z, \eta))_u.$$ 
is continuous, while the composition map ${\cal C}(X, Y)_u \times {\cal C}(Y, Z)_u \lra {\cal C}(X, Z)_u$ is not necessarily continuous. 

Let $\E(X, Y)$ and $\E^u((X, d), (Y, \rho))$ denote the space of embeddings $f : X \to Y$ 
and the subspace of uniform embeddings $f : (X, d) \to (Y, \rho)$ (both with the sup-metric and the uniform topology). 
When $X \subset Y \subset Z$, for a subset $C$ of $Z$ we use the symbol 
$\E(X, Y; C)$ to denote the subspace $\{ f \in \E(X, Y) \mid f = \id \ \text{on} \ X \cap C \}$ and  
for $\e > 0$ let $\E(i_X, \e; X, Y; C)$ denote the closed $\e$-neighborhood of the inclusion $i_X : X \subset Y$ in the space $\E(X, Y; C)$. 
%For a subset $A$ of $X$ let $\E_A(X, Y) = \{ f \in \E(X, Y) \mid f|_A = \id_A \}$. 
The meaning of the symbols $\E^u(X, Y; C)$, $\E^u(i_X, \e; X, Y; C)$, etc are obvious. 

Similarly, for a subset $A$ of $X$ 
let ${\cal H}_A(X)$ denote the group of homeomorphisms $h$ of $X$ onto itself with $h|_A = \id_A$ 
%endowed with the sup-metric and the uniform topology, 
and ${\cal H}^u_A(X, d)$ denote the subgroup of ${\cal H}_A(X)$ consisting of uniform homeomorphisms of $(X, d)$ 
 (both with the sup-metric and the uniform topology). 
We denote by ${\cal H}^u_A(X, d)_0$ the connected component of the identity $\id_X$ in ${\cal H}^u_A(X, d)$ and define 
the subgroup 
$${\cal H}^u_A(X, d)_b = \{ h \in {\cal H}^u_A(X, d) \mid d(h, \id_X) < \infty \}.$$ 
Then ${\cal H}^u_A(X, d)$ is a topological group and ${\cal H}^u_A(X, d)_b$ is an open (and closed) subgroup of ${\cal H}^u_A(X, d)$, so that 
${\cal H}^u_A(X, d)_0 \subset {\cal H}^u_A(X, d)_b$. 

The next lemma follows directly from the definitions. 

\begin{lemma}\label{lemma_unif-emb} 
For any $f \in {\cal C}(X, Y)$ the following conditions are equivalent: 
\begin{enumerate}
\item $f \in \E(X, Y)$ and $f^{-1} : (f(X), \rho) \to (X, d)$ is uniformly continuous. 
\item for any $\e > 0$ there exists $\delta > 0$ such that if $x, x' \in X$ and $d(x, x') \geq \e$ then $\rho(f(x), f(x')) \geq \delta$. 
\end{enumerate}
\end{lemma}

Recall that a family $f_\lambda \in {\cal C}(X, Y)$ $(\lambda \in \Lambda)$ is said to be equi-continuous if 
for any $\e > 0$ there exists $\delta > 0$ such that for any $\lambda \in \Lambda$ 
if $x, x' \in X$ and $d(x,x')< \delta$ then $\rho(f_\lambda(x), f_\lambda(x')) < \e$. 
More generally, we say that a family of maps $\{ f_\lambda : (X_\lambda, d_\lambda) \to (Y_\lambda, \rho_\lambda) \}_{\lambda \in \Lambda}$ between metric spaces is equi-continuous if 
for any $\e > 0$ there exists $\delta > 0$ such that for any $\lambda \in \Lambda$ 
if $x, x' \in X_\lambda$ and $d_\lambda(x,x')< \delta$ then $\rho_\lambda(f_\lambda(x), f_\lambda(x')) < \e$. 
For embeddings, we also use the following terminology: a family of embeddings 
$\{ h_\lambda : (X_\lambda, d_\lambda) \to (Y_\lambda, \rho_\lambda)\}_{\lambda \in \Lambda}$ is equi-uniform 
if both of the families $\{ h_\lambda : (X_\lambda, d_\lambda) \to (Y_\lambda, \rho_\lambda)\}_{\lambda \in \Lambda}$ and $\{(h_\lambda)^{-1} : (h_\lambda(X_\lambda), \rho_\lambda) \to (X_\lambda, d_\lambda)\}_{\lambda \in \Lambda}$ are equi-continuous. 

For a subset ${\cal C}$ of ${\cal C}(X, Y)$, the symbol 
$cl_u\,{\cal C}$ means the closure of ${\cal C}$ in ${\cal C}(X, Y)_u$.

\begin{lemma}\label{lemma_equi-conti} 
{\rm (1)} $cl_u\,\E^u(X, Y) \subset {\cal C}^u(X, Y)$. 
\begin{enumerate}
\item[(2)] Suppose ${\cal C} \subset \E^u(X, Y)$.  
If ${\cal C}' = \{ f^{-1} : f(X) \to X \mid f \in {\cal C} \}$ is equi-continuous, then $cl_u\,{\cal C} \subset \E^u(X, Y)$. 
\end{enumerate}
\end{lemma} 

\begin{proof} (1) Given $f \in cl_u\,\E^u(X, Y)$. To see that $f$ is uniformly continuous, take any $\e > 0$. 
Choose $g \in \E^u(X, Y)$ with $\rho(f,g) < \e/3$. 
Since $g$ is uniformly continuous, there exists $\delta > 0$ such that 
if $x,y \in X$ and $d(x,y) < \delta$ then $\rho(g(x), g(y)) < \e/3$. 
It follows that if $x,y \in X$ and $d(x,y) < \delta$ then
$$\rho(f(x), f(y)) \leq \rho(f(x), g(x)) + \rho(g(x), g(y)) + \rho(g(y), f(y)) < \e.$$ 

(2) Given $f \in cl_u\,{\cal C}$. By (1) $f$ is uniformly continuous. Take any $\e > 0$. 
Since ${\cal C}'$ is equi-continuous, there exists $\delta > 0$ such that 
if $g \in {\cal C}$, $x, y \in X$ and $d(x,y) \geq \e$ then $\rho(g(x), g(y)) \geq 3\delta$. 
Choose $h \in {\cal C}$ with $\rho(f, h) < \delta$. 
It follows that if $x, y \in X$ and $d(x,y) \geq \e$ then 
$$\rho(f(x), f(y)) \geq \rho(h(x), h(y)) - \rho(f(x), h(x)) - \rho(f(y), h(y)) \geq \delta.$$ 
By Lemma~\ref{lemma_unif-emb} this means that $f \in \E^u(X, Y)$. 
\end{proof} 

\begin{lemma}\label{lemma_unif_nbd}
Suppose $A$ is a compact subset of $X$ and $f \in {\cal C}(X, Y)$.
Assume that $\e, \delta >0$ satisfy the following condition: if $x,y \in A$ and $d(x,y) \leq \delta$, then $\rho(f(x), f(y)) < \e$. 
Then there exists an open neighborhood $U$ of $A$ in $X$ such that if $x,y \in U$ and $d(x,y) \leq \delta$, then $\rho(f(x), f(y)) < \e$.
\end{lemma} 

\begin{proof} We proceed by contradiction. Suppose there does not exist such an open neighborhood $U$. 
Then for each $n \geq 1$ there exists a pair of points $x_n, y_n \in C_{1/n}(A)$ such that $d(x_n, y_n) \leq \delta$ and $\rho(f(x_n), f(y_n)) \geq \e$. 
Choose points $x_n', y_n' \in A$ with $d(x_n, x_n') \leq 1/n$ and $d(y_n, y_n') \leq 1/n$. 
Since $A$ is compact, we can find subsequences $x_{n_i}'$ and $y_{n_i}'$ such that $x_{n_i}' \to x$, $y_{n_i}' \to y$ $(n \to \infty)$ in $A$. 
Then $x_{n_i} \to x$, $y_{n_i} \to y$ $(n \to \infty)$ in $X$ and so $d(x,y) \leq \delta$ and $\rho(f(x), f(y)) \geq \e$. 
This contradicts the assumption. 
\end{proof} 

%A simple example shows that Lemma~\ref{lemma_unif_nbd} does not hold 
%if the condition $d(x,y) \leq \delta$ (which appears both in the assumption and the conclusion) is   
%replaced by the condition $d(x,y) < \delta$. 

\begin{lemma}\label{lemma_conti} Suppose $P$ is a topological space, 
$f : P \to {\cal C}(X, Y)_u$, $g : P \to {\cal C}(X, Z)_u$ are continuous maps 
and $h : P \to {\cal C}^u(Y, Z)_u$ is a function. 
If $f_p$ is surjective and $h_pf_p = g_p$ for each $p \in P$, then $h$ is continuous. 
\end{lemma}

\begin{proof} Given any point $p \in P$ and any $\e > 0$. 
Since $h_p$ is uniformly continuous, there exists $\delta > 0$ such that 
if $y_1, y_2 \in Y$ and $d_Y(y_1, y_2) < \delta$, then $d_Z(h_p(y_1), h_p(y_2)) < \e/2$.  
Since $f, g$ are continuous, there exists a neighborhood $U$ of $p$ in $P$ such that 
$d_Y(f_p, f_q) < \delta$ and $d_Z(g_p, g_q) < \e/2$ for each $q \in U$. 
Then for each $q \in U$ it follows that 
$$d_Z(h_q, h_p) = d_Z(h_qf_q, h_pf_q) \leq d_Z(g_q, g_p) + d_Z(h_pf_p, h_pf_q) < \e/2 + \e/2 = \e.$$ 
\vskip -7mm 
\end{proof}

\begin{lemma}\label{lemma_collar} 
Suppose $S$ is a compact subset of $X$ which has an open collar neighborhood $\theta : (S \times [0, 4), S \times \{ 0 \}) \approx (N, S)$ in $X$. Let $N_a = \theta(S \times [0,a])$ $(a \in [0,4))$. Then there exists a strong deformation retraction 
$\phi_t$ $(t \in [0,1])$ of ${\cal H}^u_{N_1}(X)_b$ onto ${\cal H}^u_{N_2}(X)_b$ such that 
$$\mbox{$\phi_t(h) = h$ \ on \ $h^{-1}(X - N_3) - N_3$ \ \ for any \ $(h,t ) \in {\cal H}^u_{N_1}(X)_b \times [0,1]$.}$$  
\end{lemma}

\begin{proof} 
Consider the map $\gamma : [0, 1] \lra {\cal C}([0, 4), [0,4))$ defined by 
\vspace*{1mm}   
$$\gamma(s)(u) 
= \left\{ 
\begin{array}[c]{ll}
\ \ 2 u & (u \in [0, 1]) \\[1.5mm]
\ds \frac{s}{1+s}(u-1) + 2 \ & (u \in [1, 2+s]) \\[2.5mm] 
\ \ u & (u \in [2+s, 4))
\end{array} \right. \hspace{5mm} (s \in [0, 1])$$ 
\vskip 1mm 
\noindent and the homotopy $\lambda : [0, 1] \times [0,1] \lra {\cal C}([0, 4), [0,4))$ defined by 
$$\lambda_t(s)(u) = (1-t)u + t \gamma(s)(u) \hspace{6mm} ((s,t) \in [0, 1] \times [0,1], u \in [0, 4)).$$

The homotopy $\lambda$ induces a pseudo-isotopy  
\begin{itemize}
\item[] $\xi : [0, 1] \times [0,1] \lra {\cal C}^u(X, X)_u$ : 
\vspace{1mm} 
$$\xi_t(s)(x) = \left\{ 
\begin{array}[c]{@{\ }ll}
\theta(z, \lambda_t(s)(u)) & (x = \theta(z,u), (z,u) \in S \times [0,4)) \\[2mm] 
\ x & (x \in X - N_3)
\end{array}\right. 
\hspace{5mm} ((s,t) \in [0, 1] \times [0,1])$$ 
\end{itemize} 
\vskip 1mm 
satisfying the following properties : 
\begin{itemize}
\item[(1)] 
\begin{itemize}
\item[(i)\ ] $\xi_0(s) = \id_X$, \hspace{5mm} (ii) \ $\xi_t(s) = \id$ on $X - N_{2+s}$ \ \ ($(s,t) \in [0, 1] \times [0,1]$), 
\item[(iii)] $\xi_t(s) \in {\cal H}^u(X)$ \ \ ($(s,t) \in [0,1] \times [0,1] - \{ (0,1) \}$). 
\end{itemize} 
\end{itemize} 

Choose a map $\alpha : {\cal H}^u_{N_1}(X)_b \to [0,1]$ such that $\alpha^{-1}(0) = {\cal H}^u_{N_2}(X)_b$. 
By (1)(iii) we can define the homotopy 
\vspace{1mm} 
$$\phi : {\cal H}^u_{N_1}(X)_b \times [0,1] \lra {\cal H}^u_{N_1}(X)_b : \ 
\phi_t(h) = 
\left\{ \begin{array}[c]{@{\,}ll}
\xi_t(\alpha(h)) \, h \,\xi_t(\alpha(h))^{-1} & (h \in {\cal H}^u_{N_1}(X)_b - {\cal H}^u_{N_2}(X)_b),  \\[2mm] 
\ h & (h \in {\cal H}^u_{N_2}(X)_b). 
\end{array} \right.$$

\vskip 1mm 
\noindent If $h \in {\cal H}^u_{N_1}(X)_b - {\cal H}^u_{N_2}(X)_b$, then 
$\xi_t(\alpha(h))^{-1}(N_1) \subset N_1$ and $h = \id$ on $N_1$, so that $\phi_t(h) = \id$ on $N_1$. 
Since $\xi_t(0)(N_2) = N_2$ and $\xi_t(0) = \id$ on $X - N_2$ by (1)(ii), it follows that 
$h \,\xi_t(0) = \xi_t(0) h$ $((h,t) \in {\cal H}^u_{N_2}(X)_b \times [0,1])$ and so  
\begin{itemize}
\item[(2)] $\phi_t(h) \, \xi_t(\alpha(h)) = \xi_t(\alpha(h)) \, h$ \ \ $((h,t) \in {\cal H}^u_{N_1}(X)_b \times [0,1])$. 
\end{itemize}
Hence, the continuity of $\phi$ follows from Lemma~\ref{lemma_conti} 
applied to the parameter space $P = {\cal H}^u_{N_1}(X)_b \times [0,1]$ and the maps 
$$\mbox{$f : P\lra {\cal C}^u(X, X)_u$ : \ $f(h,t) = \xi_t(\alpha(h))$ \ \ and \ \ $g : P\lra {\cal C}^u(X, X)_u$ : \ $g(h, t) = \xi_t(\alpha(h)) h$.}$$ 
For each $h \in {\cal H}^u_{N_1}(X)_b - {\cal H}^u_{N_2}(X)_b$ we have 
$\xi_1(\alpha(h))^{-1}(N_2) = N_1$, so $\phi_1(h) = \id$ on $N_2$. 
These observations imply that $\phi$ is a strong deformation retraction of ${\cal H}^u_{N_1}(X)_b$ onto ${\cal H}^u_{N_2}(X)_b$.
Finally, the defining property $\xi_t(s) = \id$ on $X - N_3$ leads to the additional property 
$\phi_t(h) = h$ on $h^{-1}(X - N_3) - N_3$. 
This completes the proof. 
\end{proof}

\subsection{Basic deformation theorem for topological embeddings in topological manifolds} \mbox{} 

Next we recall the basic deformation theorem on embeddings of a compact subset in topological manifold. 
Suppose $M$ is a topological $n$-manifold possibly with boundary and $X$ is a subspace of $M$. 
An embedding $f : X \to M$ is said to be 
\begin{itemize}
\item[(i)\ ] proper if $f^{-1}(\partial M) = X \cap \partial M$ and 
\item[(ii)\,] quasi-proper if $f(X \cap \partial M) \subset \partial M$. 
\end{itemize}
For any subset $C \subset M$, let 
${\mathcal E}_\ast(X, M; C)$ and ${\mathcal E}_\#(X, M; C)$ denote the subspaces of ${\mathcal E}(X, M; C)$ consisting of proper embeddings and quasi-proper embeddings respectively. 
%(When these spaces are endowed with the compact-open topology, we use the symbols ${\mathcal E}_\ast(X, M; C)_{co}$ and ${\mathcal E}_\#(X, M; C)_{co}$.) 
Note that ${\mathcal E}_\#(X, M; C)$ is closed in ${\mathcal E}(X, M; C)$ 
(while this does not necessarily hold for ${\mathcal E}_\ast(X, M; C)$) and that 
for any $f \in {\mathcal E}_\#(X, M; C)$ the restriction of $f$ to ${\rm Int}_M X$ is a proper embedding. 
These properties are the reasons why we introduce the space of quasi-proper embeddings. 
In fact, in Section 3 we need to consider the closures of 
some collections of proper embeddings when we apply the Arzela-Ascoli theorem. 

\begin{theorem}\label{thm_basic_deform} $($\cite[Theorem 5.1]{EK}$)$  
Suppose $M$ is a topological $n$-manifold possibly with boundary,  
$C$ is a  compact subset of $M$, 
$U$ is a neighborhood of $C$ in $M$ and 
$D$ and $E$ are two closed subsets of $M$ such that $D \subset {\rm Int}_M E$.
Then, for any compact neighborhood $K$ of $C$ in $U$,  
there exists a neighborhood $\U$ of $i_U$ in ${\mathcal E}_\ast(U, M; E)$ 
and a homotopy \ 
$\varphi : \U \times [0, 1] \lra {\mathcal E}_\ast(U, M; D)$ \  such that 
\begin{itemize}
\item[{\rm (1)}]  for each $f \in {\cal U}$, 
\begin{tabular}[t]{c@{\ \,}l}
{\rm (i)} & $\varphi_0(f) = f$, \hspace{3mm} {\rm (ii)} 
$\varphi_1(f)|_C = i_C$, \hspace{3mm} {\rm (iii)} 
$\varphi_t(f) = f$ on $U - K$ $(t \in [0,1])$, \\[2mm] 
{\rm (iv)} & if $f = \id$ on $U \cap \partial M$, then $\varphi_t(f) = \id$ on $U \cap \partial M$ $(t \in [0,1])$, 
\end{tabular}
\vskip 1.5mm 
\item[{\rm (2)}] $\varphi_t(i_U) = i_U$ $(t \in [0,1])$. 
\end{itemize}
\end{theorem} 

\begin{remark}
In \cite{EK} the spaces ${\mathcal E}_\ast(U, M; E)$ and ${\mathcal E}_\ast(U, M; D)$ are endowed with the compact-open topology. Even if we replace the compact-open topology with the uniform topology, ${\cal U}$ is still a neighborhood of $i_U$ and the homotopy $\phi$ is continuous since 
the deformation is supported in the compact subset $K$ by the condition (1)(iii). 
\end{remark}

\begin{compl}\label{compl} Theorem~\ref{thm_basic_deform} still holds 
if we replace the spaces of proper embeddings, 
${\mathcal E}_\ast(U, M; D)$ and ${\mathcal E}_\ast(U, M; E)$,  
by the spaces of quasi-proper embeddings, 
${\mathcal E}_\#(U, M; D)$ and ${\mathcal E}_\#(U, M; E)$. 
\end{compl}

In fact, the quasi-proper case is derived from the proper case by the following observation. 
First we apply the proper case to ${\rm Int}_M\,U$ instead of $U$ itself, so to obtain the deformation $\phi_t$ of proper embeddings of ${\rm Int}_M\,U$. 
If $h \in {\mathcal E}_\#(U, M; E)$ is close to $i_U$, then 
we obtain the deformation $\phi_t(h|_{{\rm Int}_M\,U})$ of the restriction $h|_{{\rm Int}_M\,U}$.
Then, the condition (1)(iii) guarantees that it extends by using $h$ itself to a deformation of $h$. 

\section{Deformation lemma for uniform embeddings} 

In this section, from the deformation theorem for embeddings of compact spaces (Theorem~\ref{thm_basic_deform}) 
we derive Theorem~\ref{thm_local_deformation}, a deformation theorem for uniform embeddings in a metric covering space over a compact manifold. When passing to the uniform case from the compact case, 
the Arzela-Ascoli theorem (\cite[Theorem 6.4]{Du}) shall play an essential role. 

\subsection{Product covering case} \mbox{}  

First we consider the product covering case.  
Throughout this subsection we work under the following assumption. 

\begin{notation}\label{notation-1} 
Suppose $\pi : (M, d) \to (N, \rho)$ is a metric covering projection and $N$ is a compact topological $n$-manifold possibly with boundary. Suppose $U$ is a connected open subset of $N$ isometrically evenly covered by $\pi$, 
$C$ is a compact subset of $U$ and 
$K$ is a compact neighborhood of $C$ in $U$. 
Suppose $W$ is a subset of $M$ and 
${\cal V}$ is a collection of connected components of $\pi^{-1}(U)$  such that $V \equiv \cup\,{\cal V} \subset W$. 
Let $X = \pi^{-1}(C) \cap V$ and $P = \pi^{-1}(K) \cap V$. 
\end{notation}

%${\mathcal E}_\#(i_L, \e; L, U; E)$ denote the closed $\e$-neighborhood of $i_L$ in ${\mathcal E}_\#(L, U; E)$
The first lemma establishes a fundamental deformation theorem for uniform embeddings in the simplest case. 

\begin{lemma}\label{lem_1} 
Suppose $D$ and $E$ are closed subsets of $N$ with $D \subset {\rm Int}_N E$ and 
$Z \subset Y$ are subsets of $M$ such that 
$Y \cap V = \pi^{-1}(E) \cap V$ and $Z \cap V = \pi^{-1}(D) \cap V$.  
Then there exists a neighborhood $\W$ of the inclusion map $i_W : W \subset M$ in $\E^u_\#(W, M; Y)$ and 
a homotopy $\phi : \W \times [0,1] \lra \E^u_\#(W, M; Z)$ such that 
\begin{itemize} 
\item[(1)] for each $h \in \W$ \\ 
\begin{tabular}[t]{c@{\ \,}l}
{\rm (i)} & $\phi_0(h) = h$, \hspace{3mm} 
{\rm (ii)} $\phi_1(h) = \id$ on $X$, \hspace{3mm} 
{\rm (iii)} $\phi_t(h) = h$ on $W - P$ \ $(t \in [0,1])$, \\[2mm] 
{\rm (iv)} & $\phi_t(h)(P) \subset V$ and \ $\phi_t(h)(V) = h(V)$ \ $(t \in [0,1])$, \\[2mm] 
{\rm (v)} & if $h = \id$ on $W \cap \partial M$, then $\varphi_t(h) = \id$ on $W \cap \partial M$ $(t \in [0,1])$, 
\end{tabular}
\vskip 1.5mm 
\item[(2)] $\phi_t(i_W) = i_W$ \ $(t \in [0,1])$.
\end{itemize} 
\end{lemma}

\begin{proof} 
Since $N$ is compact, by Lemma~\ref{lemma_covering_proj}\,(1) there exists $\lambda > 0$ such that each fiber of $\pi$ is $\lambda$-discrete. 
Choose a compact subset $L$ of $U$ such that $K \subset {\rm Int}_N\, L$ and set $Q = \pi^{-1}(L) \cap V$. 
We can find $\delta \in (0, \lambda/2)$ such that $C_{\delta}(L) \subset U$ and $C_{\delta}(K) \subset {\rm Int}_N\, L$.
Let $\{ V_i \}_{i \in \Lambda}$ be the collection of connected components of $V$   
and set 
$$(Q_i, P_i, X_i) = (Q, P, X) \cap V_i$$ 
for each $i \in \Lambda$. 
Then the restriction $\pi_i := \pi|_{V_i} : (V_i, d) \to (U, \rho)$ is an isometry 
and $C_{\delta}(Q_i) \subset V_i$ since 
$d(M - V_i, Q_i) \geq \min \{ \lambda/2, \rho(N - U, L) \} > \delta$ by Lemma~\ref{lemma_covering_proj}(2).

By Deformation Theorem~\ref{thm_basic_deform} and Complement to Theorem~\ref{thm_basic_deform} (with replacing $(M, U)$ by $(U, L)$) 
there exists a neighborhood $\U$ of $i_L$ in ${\mathcal E}_\#(L, U; E)$ 
and a homotopy 
$\psi : \U \times [0, 1] \lra {\mathcal E}_\#(L, U; D)$ \  such that 
\begin{itemize}
\item[{\rm (1)}] for each $f \in {\cal U}$ \ 
\begin{tabular}[t]{c@{\ \,}l}
{\rm (i)} & $\psi_0(f) = f$, \hspace{3mm} {\rm (ii)} 
$\psi_1(f)|_C = i_C$, \hspace{3mm} {\rm (iii)} 
$\psi_t(f) = f$ on $L - K$ $(t \in [0,1])$, \\[2mm] 
{\rm (iv)} & if $f = \id$ on $L \cap \partial N$, then $\psi_t(f) = \id$ on $L \cap \partial N$ $(t \in [0,1])$, 
\end{tabular}
\vskip 1.5mm 
\item[{\rm (2)}] $\psi_t(i_L) = i_L$ $(t \in [0,1])$. 
\end{itemize}
We may assume that ${\cal U} = {\mathcal E}_\#(i_L, \gamma; L, U; E)$ (the closed $\gamma$-neighborhood of $i_L$ in $\E_\#(L, U; E)$) for some $\gamma \in (0, \delta)$. 
  
For each $i \in \Lambda$ 
we obtain the isometry with respect to the sup-metrics 
$$\theta_i : \E_\#(Q_i, V_i) \cong \E_\#(L, U) : \hspace{3mm} \theta_i(f) = \pi_i f \pi_i^{-1},$$
which restricts to the isometries 
$$\theta_i' : \E_\#(Q_i, V_i; Y) \cong \E_\#(L, U; E) \hspace{5mm} and \hspace{5mm} 
\theta_i'' : \E_\#(Q_i, V_i; Z) \cong \E_\#(L, U; D).$$ 
Then, ${\cal W}_i \equiv (\theta_i')^{-1}({\cal U}) = {\mathcal E}_\#(i_{Q_i}, \gamma; Q_i, V_i; Y)$ 
%(the $\e$-neighborhood of $i_{P_i}$ in $\E_\#(P_i, V_i; Y)$) 
and the homotopy $\psi$ induces the corresponding homotopy 
$$\phi^i : {\cal W}_i \times [0,1] \to \E_\#(Q_i, V_i; Z),$$
which satisfies the following conditions
\begin{itemize}
\item[{\rm (3)}] for each $f \in {\cal W}_i$ \\
\begin{tabular}[t]{c@{\ \,}l}
{\rm (i)} & $\varphi^i_0(f) = f$, \hspace{3mm} {\rm (ii)} 
$\varphi^i_1(f) = \id$ \ on \ $X_i$, \hspace{3mm} {\rm (iii)} 
$\varphi^i_t(f) = f$ \ on \ $Q_i - P_i$ \ $(t \in [0,1])$, \\[2mm] 
{\rm (iv)} & if $f = \id$ on $Q_i \cap \partial M$, then $\varphi^i_t(f) = \id$ on $Q_i \cap \partial M$ $(t \in [0,1])$, 
\end{tabular}
\vskip 1.5mm 
\item[{\rm (4)}] $\varphi^i_t(i_{Q_i}) = i_{Q_i}$ \ $(t \in [0,1])$. 
\end{itemize}
 
Let ${\cal W} = \E^u_\#(i_W, \gamma; W, M; Y)$ and define a homotopy $\phi : {\cal W} \times [0,1] \lra \E^u_\#(W, M; Z)$ as follows. 
Take any $h \in {\cal W}$. Since $\gamma < \delta$,
for any $i \in \Lambda$ we have $h(Q_i) \subset C_{\delta}(Q_i) \subset V_i$ and 
$h|_{Q_i} \in {\cal W}_i$. 
Therefore we can define $\phi_t(h)$ $(t \in [0,1])$ by 
$$\phi_t(h)|_{Q_i} = \phi^i_t(h|_{Q_i}) \hspace{3mm} (i \in \Lambda) \hspace{5mm} 
\text{and} \hspace{5mm} \phi_t(h) = h \ \ \text{on} \ \ W - P.$$ 
 
Since $\phi^i_t(h|_{Q_i}) = h$ on $Q_i - P_i$, the map $\phi_t(h)$ is a well-defined embedding 
and the required conditions (1), (2) for $\phi$ follow from the corresponding conditions (3), (4) for $\phi^i$ $(i \in \Lambda)$. 
For (1)(iv) note that $\phi_t(h)(Q_i) = \phi^i_t(h|_{Q_i})(Q_i) = h(Q_i)$ $(i \in \Lambda)$. 
It remains to show that 
\begin{itemize}
\item[] $(\ast)_1$ $\phi_t(h)$ is a uniform embedding for any $h \in {\cal W}$ and $t \in [0,1]$ 
\hspace{1mm} and 
\hspace{1mm} $(\ast)_2$ $\phi$ is continuous. 
\end{itemize}

Take any $h \in {\cal W}$. For each $i \in \Lambda$ let $h_i = \theta_i'(h|_{Q_i})$. 
Since $h$ is a uniform embedding, 
the family $h|_{Q_i} \in {\cal W}_i$ ($i \in \Lambda$) is an equi-uniform family of embeddings.  
Therefore, the families ${\cal C}(h) = \{ h_i \}_{i \in \Lambda}$ and ${\cal C}'(h) = \{ h_i^{-1} \}_{i \in \Lambda}$ are also equi-continuous. Since ${\rm Im}\,h_i \subset C_{\delta}(L) \subset U$ $(i \in \Lambda)$ and $C_{\delta}(L)$ is compact,  
by the Arzela-Ascoli theorem (\cite[Theorem 6.4]{Du}) the closure $cl\,{\cal C}(h)$ of ${\cal C}(h)$ in ${\cal C}(L, U)$ is compact. 
It also follows that $cl\,{\cal C}(h) \subset {\cal U} \subset \E_\#(L, U; E)$ 
by Lemma~\ref{lemma_equi-conti} and the equi-continuity of ${\cal C}'(h)$.

Now we show that 
$\psi_t(h_i) \in \E_\#(L, U; D)$ \ $(i \in \Lambda, t \in [0,1])$ \ is an equi-uniform family of embeddings.   
Since $\psi(cl\,{\cal C}(h) \times [0,1]) \subset \E_\#(L, U; D)$ is compact, 
the family $\psi_t(h_i)$ $(i \in \Lambda, t \in [0,1])$ is equi-continuous.  
The equi-continuity of the family $(\psi_t(h_i))^{-1}$ $(i \in \Lambda, t \in [0,1])$ is shown as follows.  Since ${\rm Im}\,\psi_t(f) = {\rm Im}\,f$ for each $(f,t) \in {\cal U} \times [0,1]$, 
we have the map  
$$\mbox{$\chi : {\cal U} \times [0,1] \lra {\cal H}(L)$ : $\chi_t(f) = (\psi_t(f))^{-1}f$.}$$ 
Since $\chi(cl\,{\cal C}(h) \times [0,1])$ is compact, 
the family $\chi_t(h_i)$ $(i \in \Lambda, t \in [0,1])$ is equi-continuous. 
Since ${\cal C}'(h) = \{ h_i^{-1} \}_{i \in \Lambda}$ is equi-continuous, 
the family 
$$(\psi_t(h_i))^{-1} = \chi_t(h_i) h_i^{-1} \ \ (i \in \Lambda, t \in [0,1])$$ 
is also equi-continuous as desired. 

$(\ast)_1$ 
Pulling back each map $\psi_t(h_i)$ by the isometry $\theta_i''$, 
it follows that $\phi^i_t(h|_{Q_i})$ $(i \in \Lambda, t \in [0,1])$ is an equi-uniform family of embeddings.
The map $\phi_t(h)$ is uniformly continuous, 
since $\phi_t(h)|_{W-P} = h|_{W-P}$ is uniformly continuous, 
the family $\phi_t(h)|_{Q_i} = \phi^i_t(h|_{Q_i})$ $(i \in \Lambda)$ is equi-continuous 
and $C_\delta(P_i) \subset Q_i$ $(i \in \Lambda)$. 
Similarly the map $\phi_t(h)^{-1}h$ is uniformly continuous, 
since $\phi_t(h)^{-1}h = \id$ on $W-P$, 
the family $\phi_t(h)^{-1}h|_{Q_i} = \phi^i_t(h|_{Q_i})^{-1}h|_{Q_i}$ $(i \in \Lambda)$ is equi-continuous 
and $C_\delta(P_i) \subset Q_i$ $(i \in \Lambda)$. 
Therefore, $\phi_t(h)^{-1} = (\phi_t(h)^{-1}h)h^{-1}$ is also uniformly continuous. 

$(\ast)_2$ To see the continuity of the homotopy $\phi$, 
take any $(h,t) \in {\cal W} \times [0,1]$ and $\e > 0$. 
Since $cl\,{\cal C}(h)$ is compact, 
the homotopy 
$$\psi : {\cal U} \times [0,1] \lra 
\E_\#(L, U; D)$$
is uniformly continuous on $cl\,{\cal C}(h) \times [0,1]$. 
Hence there exists $\eta \in (0, \e)$ such that 
\begin{itemize}
\item[] 
if $(f, u), (g, v) \in cl\,{\cal C}(h) \times [0,1]$ and $\rho(f,g) \leq \eta$, $|u-v| \leq \eta$, then $\rho(\psi_u(f),\psi_v(g)) < \e$.  
\end{itemize}
By Lemma~\ref{lemma_unif_nbd} we can find a neighborhood ${\cal O}$ of $cl\,{\cal C}(h)$ in 
${\cal U}$ such that 
\begin{itemize}
\item[] if $(f, u), (g, v) \in {\cal O} \times [0,1]$ and $\rho(f,g) \leq \eta$, $|u-v| \leq \eta$, then $\rho(\psi_u(f),\psi_v(g)) < \e$. 
\end{itemize} 
Choose $\zeta \in (0, \eta)$ such that 
${\cal O}$ contains the open $\zeta$-neighborhood ${\cal O}(\zeta)$ of $cl\,{\cal C}(h)$ in ${\cal U}$. 
Then it follows that  
\begin{itemize}
\item[] if $(k,s) \in {\cal W} \times [0,1]$ and $d(h,k) < \zeta$, $|t-s| < \zeta$, then 
$d(\phi_t(h),\phi_s(k)) \leq \e$. 
\end{itemize} 
In fact, take any $(k,s) \in {\cal W} \times [0,1]$ with $d(h,k) < \zeta$ and $|t-s| < \zeta$. 
Then, for any $i \in \Lambda$, 
we have $\rho(h_i, k_i) = d(h|_{Q_i}, k|_{Q_i}) < \zeta$ and $h_i \in {\cal C}(h)$,  
so that $h_i, k_i \in {\cal O}(\zeta) \subset {\cal O}$. 
Hence, by the choice of ${\cal O}$ and $\eta$ 
we have $\rho(\psi_t(h_i), \psi_s(k_i)) < \e$ and so 
$d(\phi^i_t(h|_{Q_i}), \phi^i_s(k|_{Q_i})) < \e$. 
This implies $d(\phi_t(h),\phi_s(k)) \leq \e$. 
This completes the proof. 
\end{proof}

The next lemma deals with the problem on the pattern of intersection of each sheet of $V$ with $Y$ 
(which is represented by a set $I_{V_0}$ defined in the proof). 
Here, $V$ represents the subset of $W$ on which we shall deform the uniform embeddings, 
while $Y$ represents the subset on which the uniform embeddings have already been deformed to the identity. 
In Lemma~\ref{lem_1} the pattern is same for all sheets of $V$ (corresponding to the inverse image of a subset of $N$), 
while in Lemma~\ref{lem_2} appear finitely many patterns of intersections 
(relating to the inverse images of finitely many subsets of $N$) and   
Lemma~\ref{lem_1} is applied to each pattern separately. 

When $O_0$ is a connected open subset of $N$ isometrically evenly covered by $\pi$, 
let ${\cal S}(O_0)$ denote the collection of connected components of $\pi^{-1}(O_0)$. 
 For each $O \in {\cal S}(O_0)$ the restriction $\pi|_O : O \to O_0$ is an isometry. 
For subsets $A_0, B_0 \subset O_0$, let 
$${\cal S}(O_0, A_0, B_0) = \{ (O, A, B) \mid O \in {\cal S}(O_0), A = (\pi|_O)^{-1}(A_0), B = (\pi|_O)^{-1}(B_0)\}.$$ 

%The next lemma deals with the pattern of intersection of each sheet of $V$ with $Y$.  
%In Lemma~\ref{lem_1} the pattern $Y \cap V$ is the inverse image in $V$ of a subset of $N$ so that 
%the pattern of intersection of each sheet of $V$ with $Y$ is same for all sheets of $V$, 
%while in Lemma~\ref{lem_2} the subset $Y$ is the union of some sort of subsets of the inverse images in $V$ of finitely many subsets of $N$ and 
%the pattern of intersection of each sheet of $V$ with $Y$ is different for .  

We keep the notations given in Notation~\ref{notation-1}. 

\begin{lemma}\label{lem_2}
Suppose $\{ (O_j, E_j, D_j) \}_{j\in J}$ is a finite family of subsets of $N$ such that 
(a) for each $j \in J$, $O_j$ is a connected open subset of $N$ and $E_j, D_j$ are closed subsets of $N$ with 
$D_j \subset {\rm Int}_N E_j$ and $E_j \subset O_j$ and 
(b) $O_j$ $(j \in J)$ and ${\rm St}(U, \{ O_j \}_{j \in J})$ are isometrically evenly covered by $\pi$.  
Suppose ${\cal F}$ is a subcollection of $\cup_{j \in J} \{ j \} \times {\cal S}(O_j, E_j, D_j)$ and 
let 
$$Y = \cup \{ E \mid (j, (O, E, D)) \in {\cal F} \} \hspace{5mm} \text{and} \hspace{5mm}   
Z = \cup \{ D \mid (j, (O, E, D)) \in {\cal F} \}.$$ 
Then there exists a neighborhood $\W$ of the inclusion map $i_W$ in $\E^u_\#(W, M; Y)$ and 
a homotopy $\phi : \W \times [0,1] \lra \E^u_\#(W, M; Z)$ such that 
\begin{itemize} 
\item[(1)] for each $h \in \W$ \\ 
\begin{tabular}[t]{c@{\ \,}l}
{\rm (i)} & $\phi_0(h) = h$, \hspace{3mm} 
{\rm (ii)} $\phi_1(h) = \id$ on $X$, \\[2mm] 
{\rm (iii)}  & $\phi_t(h) = h$ \ on \ $W - P$ \ \ and \ \ $\phi_t(h)(W) = h(W)$ \ \ $(t \in [0,1])$, \\[2mm] 
{\rm (iv)} & if $h = \id$ on $W \cap \partial M$, then $\varphi_t(h) = \id$ on $W \cap \partial M$ $(t \in [0,1])$, 
\end{tabular} 
\vskip 1.5mm 
\item[(2)] $\phi_t(i_W) = i_W$ \ $(t \in [0,1])$.
\end{itemize} 
\end{lemma} 

\begin{proof} 
For each $V_0 \in {\cal V}$
consider the subset $I_{V_0}$ of $J$ defined by 
$$I_{V_0} = \{ j \in J \mid \exists \, (j, (O, E, D)) \in {\cal F} \ \text{such that} \ E \cap V_0 \neq \emptyset \}.$$ 
For each $I \subset J$ let 
$$\begin{array}[t]{l}
{\cal V}_I = \{ V_0 \in {\cal V} \mid I_{V_0} = I \}, \ \ 
V_I \equiv \cup \,{\cal V}_I, \ \ X_I = \pi^{-1}(C) \cap V_I, \ \ P_I = \pi^{-1}(K) \cap V_I \\[2mm]  
E_I = \cup_{i \in I} E_i \ \ \text{and} \ \ D_I = \cup_{i \in I} D_i.
\end{array}$$ 
\vskip 1mm 
\noindent It follows that $V_I \subset V \subset W$ and $D_I \subset {\rm Int}_N E_I$. 
We show that 
$${\rm (i)} \ \ Y \cap V_I = \pi^{-1}(E_I) \cap V_I \hspace{5mm} \text{and} \hspace{5mm} {\rm (ii)} \ \ Z \cap V_I = \pi^{-1}(D_I) \cap V_I.$$
%Since $$\mbox{$V_I \subset V \subset W$, \ \ $D_I \subset {\rm Int}_N E_I$, \ \ 
%$Y \cap V_I = \pi^{-1}(E_I) \cap V_I$ \ \ and \ \ $Z \cap V_I = \pi^{-1}(D_I) \cap V_I$,}$$  

(i) Given $x \in Y \cap V_I$. 
Since $x \in Y$, it follows that 
$x \in E$ for some $(i, (O,E,D)) \in {\cal F}$, thus $(O,E,D) \in {\cal S}(O_i,E_i,D_i)$ and $E = (\pi|_O)^{-1}(E_i)$, so 
$\pi(x) \in E_i$. 
Since $x \in V_I$, it is seen that $x \in V_0$ for some $V_0 \in {\cal V}_I$ and $I_{V_0} = I$. 
%It follows that $x \in E$ for some $(i, (O,E,D)) \in {\cal F}$ and $x \in V_0$ for some $V_0 \in {\cal V}_I$, 
%which means that $I_{V_0} = I$. 
Since $x \in E \cap V_0 \neq \emptyset$, we have $i \in I_{V_0} = I$ and $E_i \subset E_I$. 
Hence $\pi(x) \in E_I$ and $x \in \pi^{-1}(E_I) \cap V_I$. 

Conversely suppose $x \in \pi^{-1}(E_I) \cap V_I$. 
Then $\pi(x) \in E_I$, thus $\pi(x) \in E_i$ for some $i \in I$. 
Since $x \in V_I$, it follows that $x \in V_0$ for some $V_0 \in {\cal V}_I$ and $I_{V_0} = I \ni i$ so that 
$E \cap V_0 \neq \emptyset$ for some $(i, (O,E,D)) \in {\cal F}$, so $(O,E,D) \in {\cal S}(O_i,E_i,D_i)$ and $E \subset Y$.  
By the assumption $\widetilde{U} \equiv {\rm St}(U, \{ O_j \}_{j \in J})$ is a connected open subset of $N$ isometrically evenly covered by $\pi$. 
Since $\pi(x) \in U \cap E_i \subset U \cap O_i$, we have $O_i \subset \widetilde{U}$, so $V_0, O \subset \pi^{-1}(\widetilde{U})$. 
Since $V_0$ and $O$ are connected and  $V_0 \cap O \supset V_0 \cap E \neq \emptyset$, 
there exists $\widetilde{U}_0 \in {\cal S}(\widetilde{U})$ with $V_0, O \subset \widetilde{U}_0$.
Since $\pi : \widetilde{U}_0 \to \widetilde{U}$ is an isometry, 
$E \subset O \subset \widetilde{U}_0$, $x \in V_0 \subset \widetilde{U}_0$ and 
$\pi(x) \in E_i \subset O_i \subset \widetilde{U}$ 
it follows that $x \in E \subset Y$ so $x \in Y \cap V_I$ as desired. 

The assertion (ii) follows from the same argument as (i). 

By Lemma~\ref{lem_1} there exists 
a neighborhood $\W_I$ of $i_W$ in $\E^u_\#(W, M; Y)$ and 
a homotopy $\phi^I : \W_I \times [0,1] \lra \E^u_\#(W, M; Z)$ such that 
\begin{itemize} 
\item[(1)] for each $h \in \W_I$ \\ 
\begin{tabular}[t]{c@{\ \,}l}
{\rm (i)} & $\phi^I_0(h) = h$, \hspace{3mm} 
{\rm (ii)} $\phi^I_1(h) = \id$ on $X_I$, \hspace{3mm} 
{\rm (iii)} $\phi^I_t(h) = h$ on $W - P_I$ \ $(t \in [0,1])$, \\[2mm] 
{\rm (iv)} & $\phi^I_t(h)(P_I) \subset V_I$ \  and \ $\phi^I_t(h)(V_I) = h(V_I)$ \ $(t \in [0,1])$, \\[2mm] 
{\rm (v)} & if $h = \id$ on $W \cap \partial M$, then $\phi^I_t(h) = \id$ on $W \cap \partial M$ $(t \in [0,1])$, 
\end{tabular}
\vskip 1.5mm 
\item[(2)] $\phi^I_t(i_W) = i_W$ \ $(t \in [0,1])$.
\end{itemize} 

Since ${\cal V}$ is the disjoint union of the subcollections ${\cal V}_I$ $(I \subset J)$, 
it follows that $V$ is the disjoint union of $V_I$ $(I \subset J)$.  
Then ${\cal W} = \cap_{I \subset J} {\cal W}_I$ is a neighborhood of $i_W$ in $\E^u_\#(W, M; Y)$ and we can define 
a homotopy $\phi : \W \times [0,1] \lra \E^u_\#(W, M; Z)$ by  
$$\mbox{$\phi_t(h) = \phi^I_t(h)$ \ \ on \ \ $V_I$ \hspace{3mm} and \hspace{3mm} $\phi_t(h) = h$ \ \ on \ \ $W - P$.}$$  
Since there exists $\gamma > 0$ such that $C_\gamma(P_I) \subset V_I$ $(I \subset J)$, 
the uniform continuity of $\phi_t(h)$ follows from those of the maps $h$ and $\phi^I_t(h)$ $(I \subset J)$. 
Similarly, the map $\phi_t(h)^{-1}h$ is uniformly continuous since $\phi_t(h)^{-1}h = \id$ on $W - P$ and 
the maps $\phi^I_t(h)^{-1}h$ $(I \subset J)$ are uniformly continuous. 
Thus $\phi_t(h)^{-1} = (\phi_t(h)^{-1}h)h^{-1}$ is also uniformly continuous. 
\end{proof} 

\subsection{General case} \mbox{} 

Theorem~\ref{thm_local_deformation} is easily deduced from Theorem~\ref{thm_local_deformation-2}, 
whose proof is based upon a recursive application of Lemma~\ref{lem_2} to a finite family of local trivializations of the metric covering projection $\pi$.  
Here, the key is to set up the correct data to which Lemma~\ref{lem_2} is applied.  
%In each inductive step, the key point is to define correctly the collection ${\cal F}$ to which Lemma~\ref{lem_2} is applied.  
%which is represented by the definition of the collection ${\cal F}_i$ appeared in the proof.
%The key point is to grasp ``the patterns of intersection'' correctly in each step, 
%which is represented by the definition of the collection ${\cal F}_i$ appeared in the proof.
%proceeds by an inductive argument with respect to a finite family of local trivializations of the metric covering projection $\pi$. 

\begin{theorem}\label{thm_local_deformation-2} 
Suppose $\pi : (M, d) \to (N, \rho)$ is a metric covering projection, $N$ is a compact topological $n$-manifold possibly with boundary,  
$X$ is a closed subset of $M$, $W' \subset W$ are uniform neighborhoods of $X$ in $(M, d)$ and 
$Z$, $Y$ are closed subsets of $M$ such that $Y$ is a uniform neighborhood of $Z$. 
Then there exists a neighborhood $\W$ of the inclusion map $i_W : W \subset M$ in $\E^u_\#(W, M; Y)$ and 
a homotopy $\phi : \W \times [0,1] \lra \E^u_\#(W, M; Z)$ such that 
\begin{itemize} 
\item[(1)] for each $h \in \W$ \\ 
\begin{tabular}[t]{c@{\ \,}l}
{\rm (i)} & $\phi_0(h) = h$, \hspace{3mm} 
{\rm (ii)} $\phi_1(h) = \id$ \ on \ $X$, \\[2mm] 
{\rm (iii)} & $\phi_t(h) = h$ \ on \ $W - W'$ \ \ and \ \ $\phi_t(h)(W) = h(W)$ \ \ $(t \in [0,1])$, \\[2mm] 
{\rm (iv)} & if $h = \id$ on $W \cap \partial M$, then $\phi_t(h) = \id$ on $W \cap \partial M$ $(t \in [0,1])$, 
\end{tabular} 
\vskip 1.5mm 
\item[(2)] $\phi_t(i_W) = i_W$ \ $(t \in [0,1])$.
\end{itemize} 
\end{theorem}

\begin{proof} For $m \in \IN$ let $[m] = \{ 1, 2, \cdots, m \}$. 
Choose $\gamma > 0$ such that $C_\gamma(X) \subset W'$ and $C_\gamma(Z) \subset Y$. 
Since $N$ is compact, there exists a finite open cover ${\cal U} = \{ U_i \}_{i\in[m]}$ of $N$ such that for each $i \in [m]$ 
\begin{itemize}
\item[] $\diam U_i < \gamma$, \ 
$U_i$ is connected \ and \ ${\rm St}\,(U_i, \U)$ is isometrically evenly covered by $\pi$. 
\end{itemize} 
There exists a finite closed covering $\F = \{ F_i \}_{i\in[m]}$ of $N$ such that $F_i \subset U_i$ for each $i \in [m]$. 

By Lemma~\ref{lemma_covering_proj} there exists $\lambda > 0$ such that each fiber of $\pi$ is $\lambda$-discrete.
Choose $\delta \in (0, \lambda/2)$ such that $C_{\delta}(F_i) \subset U_i$ for each $i \in [m]$. 
Take real numbers
$$\delta > \delta_0 > \delta_1 > \cdots > \delta_m > 0.$$ 
For each $i \in [m]$ we apply Lemma~\ref{lem_2} to the following data: 
\vskip 2mm 
\begin{tabular}[t]{l}
$U = U_i \subset N$, \hspace{5mm} $(K_i, C_i) = (C_{\delta_{i-1}}(F_i), C_{\delta_{i}}(F_i))$, \hspace{5mm} $W \subset M$, \\[3mm] 
${\cal V}_i = \{ V' \in {\cal S}(U_i) \mid V' \cap X \neq \emptyset \}$, 
\hspace{5mm} $V_i = \cup \, {\cal V}_i$, 
\hspace{5mm} $X_i = \pi^{-1}(C_i) \cap V_i$, 
\hspace{5mm} $P_i = \pi^{-1}(K_i) \cap V_i$, \\[3mm] 
$(O_j, E_j^i, D_j^i) = (U_j, C_{\delta_{i-1}}(F_j), C_{\delta_{i}}(F_j))$ \ $(j \in [m])$, \\[3mm] 
${\cal F}_i = \{(k, (O, E, D)) \in \bigcup_{j \in [m]} \{ j \} \times {\cal S}(O_j, E_j^i, D_j^i) \mid 
\mbox{(a) $E \cap X \neq \emptyset$ and $k \leq i-1$ or (b) $E \cap Z \neq \emptyset$} \}.$ \\[3mm] 
$Y_i = \cup \{ E \mid (k, (O, E, D)) \in {\cal F}_i\}$ \hspace{5mm} and \hspace{5mm} 
$Z_i = \cup \{ D \mid (k, (O, E, D)) \in {\cal F}_i\}$. 
\end{tabular} \\[3mm] 
By the choice of $\gamma$ it is seen that $V_i \subset C_\gamma(X) \subset W' \subset W$. 
Thus we obtain a neighborhood $\W_i$ of $i_W$ in $\E^u_\#(W, M; Y_i)$ and 
a homotopy $\phi^i : \W_i \times [0,1] \lra \E^u_\#(W, M; Z_i)$ such that 
\begin{itemize} 
\item[(1)] for each $h \in \W_i$ \\ 
\begin{tabular}[t]{c@{\ \,}l}
{\rm (i)} & $\phi^i_0(h) = h$, \hspace{3mm} 
{\rm (ii)} $\phi^i_1(h) = \id$ \ on \ $X_i$, \\[2mm] 
{\rm (iii)} & $\phi^i_t(h) = h$ \ on \ $W - P_i$ \ \ and \ \ $\phi^i_t(h)(W) = h(W)$ \ \ $(t \in [0,1])$, \\[2mm] 
{\rm (iv)} & if $h = \id$ on $W \cap \partial M$, then $\phi_t^i(h) = \id$ on $W \cap \partial M$ $(t \in [0,1])$, 
\end{tabular} 
\vskip 1.5mm 
\item[(2)] $\phi^i_t(i_W) = i_W$ \ $(t \in [0,1])$.
\end{itemize} 

To compose these homotopies we use the following implications; 
\begin{itemize}
\item[(3)] $Y_{i+1} \subset Z_i \cup X_i$ \ \ $(i \in [m-1])$, \hspace{5mm}
\item[(4)] (i) $Z \subset Z_i$ \ \ $(i \in [m])$, \hspace{5mm} 
(ii) $X \subset X_m \cup Z_m$, \hspace{5mm} 
(iii) \,$Y_1 \subset Y$. %$X \cup Z \subset X_m \cup Z_m$, \hspace{5mm}  
%(iv) $Y_1 \subset Y$. 
\end{itemize}

We will verify these statements later and continue the construction of the required homotopy $\phi$. 
By (3) and (1)(ii) we have the maps $\phi^i_1 : \W_i \to \E^u_\#(W, M; Y_{i+1})$ $(i \in [m-1])$.
Since $\phi^i_1(i_W) = i_W \in \W_{i+1}$, 
by the backward induction the neighborhoods $\W_i$ $(i \in [m-1])$ can be replaced by smaller ones 
so to achieve the condition $\phi^i_1(\W_i) \subset \W_{i+1}$. 
Since $\E^u_\#(W, M; Y) \subset \E^u_\#(W, M; Y_1)$ by (4)(iii), 
there exists a neighborhood $\W$ of $i_W$ in $\E^u_\#(W, M; Y)$ such that $\W \subset \W_1$. 
Then we have the composition maps $\phi^{i-1}_1 \cdots \phi^1_1 : \W \to \W_i$ $(i \in [m])$, 
where $\phi^{i-1}_1 \cdots \phi^1_1 = i_{\W}$ for $i=1$.  
Finally, since $\E^u_\#(W, M; Z_i) \subset \E^u_\#(W, M; Z)$ by (4)(i), 
we can define the required homotopy 
$$\mbox{$\phi  : \W \times [0,m] \lra \E^u_\#(W, M; Z)$ \hspace{2mm} by 
\hspace{2mm}  
$\phi_t = \phi^i_{t-i+1}\phi^{i-1}_1 \cdots \phi^1_1$ \hspace{2mm} $(t \in [i-1,i]$, $i \in [m])$.}$$  
By (1)(i) the homotopy $\phi$ is well-defined and the required conditions (1), (2) for $\phi$ follow from the corresponding properties (1), (2) of the homotopies $\phi^i$ $(i \in [m])$. 
For (1) (ii) note that $\phi_m(h) = \phi_1^m(\phi^{m-1}_1 \cdots \phi^1_1(h)) = \id$ on $X_m \cup (W \cap Z_m)$ and 
that $X \subset W \cap (X_m \cup Z_m) = X_m \cup (W \cap Z_m)$ by (4)(ii). 

It remains to verify the assertions (3) and (4). 

(3) Take any $y \in Y_{i+1}$. 
We have $y \in E$ for some $(k, (O,E,D)) \in {\cal F}_{i+1}$, so  
$(O,E,D) \in {\cal S}(O_k, E_k^{i+1}, D_k^{i+1})$ and 
 (a) $E \cap X \neq \emptyset$ and $k \leq i$ or (b) $E \cap Z \neq \emptyset$. 
 It follows that $\pi|_O : O \to O_k$ is an isometry and $E = (\pi|_O)^{-1}(D_k^i)$ since $E_k^{i+1} = D_k^i$. 

In the case (a) with $k = i$;  
Since $O_k = U_i$ and $D_k^i = D_i^i = C_i$, it follows that 
$y \in E = (\pi|_O)^{-1}(C_i) \subset \pi^{-1}(C_i)$. 
Since $O \cap X \supset E \cap X \neq \emptyset$, it follows that 
$O \in {\cal V}_i$ and $V_i \supset O \supset E \ni y$. 
Hence we have $y \in \pi^{-1}(C_i) \cap V_i = X_i$ 

In the case (a) with $k \leq i-1$ or (b); 
Let $E' = (\pi|_O)^{-1}(E_k^i)$. Then $(O, E', E) \in {\cal S}(O_k, E_k^i, D_k^i)$ and  
$(k, (O, E', E)) \in {\cal F}_i$ since $E' \cap X \supset E \cap X \neq \emptyset$ in the case (a) with $k \leq i-1$ and 
$E' \cap Z \supset E \cap Z \neq \emptyset$ in the case (b). 
Hence we have $y \in E \subset Z_i$.  

(4)(ii) Give any $x \in X$. Then $\pi(x) \in F_k$ for some $k \in [m]$. 

In the case where $k \leq m - 1$; 
Since $\pi(x) \in F_k \subset U_k = O_k$, it follows that 
$x \in O$ for some $O \in {\cal S}(O_k)$ and $\pi|_O : O \to O_k$ is an isometry.
Put $E = (\pi|_O)^{-1}(E_k^m)$ and $D = (\pi|_O)^{-1}(D_k^m)$. 
Then, it follows that $x \in D$ since $\pi(x) \in F_k \subset D_k^m$, and that 
$(k, (O, E, D)) \in {\cal F}_m$ since $(O, E, D) \in {\cal S}(O_k, E_k^m, D_k^m)$, $x \in E \cap X \neq \emptyset$ and $k \leq m-1$. 
This implies that $x \in D \subset Z_m$. 

In the case where $k = m$; 
Since $\pi(x) \in F_m \subset C_m \subset U_m$, 
it follows that $x \in \pi^{-1}(C_m)$ and there exists $V' \in {\cal S}(U_m)$ with $x \in V'$. 
Since $x \in V' \cap X$, we have $V' \in {\cal V}_m$ and $x \in V' \subset V_m$. 
This implies that $x \in \pi^{-1}(C_m) \cap V_m = X_m$.  

The statements (4)(i) and (4)(iii) are verified similarly. This completes the proof. 
\end{proof} 

\section{Groups of uniform homeomorphisms of metric spaces with bi-Lipschitz Euclidean ends} 

In this section we study some global deformation properties of groups of uniform homeomorphisms of manifolds with bi-Lipschitz Euclidean ends. 
The Euclidean space $\IR^n$ admits the canonical Riemannian covering projection $\pi : \IR^n \to \IR^n/\IZ^n$ onto the flat torus. 
Therefore we can apply the Local Deformation theorem Theorem~\ref{thm_local_deformation} to uniform embeddings in $\IR^n$. 

\begin{proposition}\label{prop_deform_Euclid} 
For any closed subset $X$ of $\IR^n$ and any uniform neighborhoods $W' \subset W$ of $X$ in $\IR^n$ 
there exists a neighborhood $\W$ of the inclusion map $i_W : W \subset \IR^n$ in $\E^u_\ast(W, \IR^n)$ and 
a homotopy $\phi : \W \times [0,1] \lra \E^u_\ast(W, \IR^n)$ such that 
\begin{itemize} 
\item[(1)] for each $h \in \W$ \ \ 
\begin{tabular}[t]{c@{\ }l}
{\rm (i)} & $\phi_0(h) = h$, \ \ 
{\rm (ii)} $\phi_1(h) = \id$ on $X$, \\[2mm] 
{\rm (iii)} & $\phi_t(h) = h$ on $W - W'$ \ \ and \ \ $\phi_t(h)(W) = h(W)$ \ \ $(t \in [0,1])$,  
\end{tabular}
\vskip 1mm 
\item[(2)] $\phi_t(i_W) = i_W$ \ $(t \in [0,1])$.
\end{itemize} 
\end{proposition} 

The relevant feature of Euclidean space $\IR^n$ in this context is the existence of similarity transformations 
$$\mbox{$k_\gamma : \IR^n \approx \IR^n$ : \ $k_\gamma(x) = \gamma x$ \hspace{5mm} \ $(\gamma > 0)$.}$$
This enables us to deduce, from the local one, a global deformation in groups of uniform homeomorphisms on $\IR^n$ 
and more generally, manifolds with bi-Lipschitz Euclidean ends. 

\subsection{Euclidean ends case} \mbox{} 

Recall our conventions: For $r \in \IR$ we set $\IR^n_r = \IR^n - O(r)$, where $O(r) = \{ x \in \IR^n \mid \| x \| < r \}$. 
For $s > r > 0$ and $\e > 0$, let $\E^u(\iota_s, \e; \IR^n_s, \IR^n_r)$ denote the open $\e$-neighborhood of the inclusion map $\iota_{s,r} : \IR^n_s \subset \IR^n_r \subset \IR^n$ in the space $\E^u(\IR^n_s, \IR^n_r)_u.$
We can apply Proposition~\ref{prop_deform_Euclid} to $(X, W', W) = (\IR^n_v, \IR^n_u, \IR^n_s)$ and replace $\W$ by a smaller one to obtain the following conclusion. 

\begin{lemma}\label{lemma_local_deform_E-end} 
For any $0 \leq r < s < u < v$ and $\e > 0$ there exist $\delta > 0$ and a homotopy 
$$\phi : \E^u(\iota_{s,r}, \delta; \IR^n_s, \IR^n_r) \times [0,1] \lra \E^u(\iota_{s,r}, \e; \IR^n_s, \IR^n_r)$$ 
such that {\rm (1)} for each $h \in \E^u(\iota_{s,r}, \delta; \IR^n_s, \IR^n_r)$ 
\begin{itemize}
\item[] \hspace*{8mm} {\rm (i)} $\phi_0(h) = h$, \ \ {\rm (ii)} $\phi_1(h) = \id$ on $\IR^n_v$, \ \ 
{\rm (iii)} $\phi_t(h) = h$ on $\IR^n_s - \IR^n_u$ \ $(t \in [0,1])$,  

\item[(2)] $\phi_t(\iota_{s,r}) = \iota_{s,r}$ \ $(t \in [0,1])$.
\end{itemize} 
\end{lemma}

Now we apply a similarity transformation $k_\gamma$ for a sufficiently large $\gamma > 0$ to Lemma~\ref{lemma_local_deform_E-end}.

\begin{lemma}\label{lemma_deformation} 
For any $c, s_0 > 0$ and $\beta > \alpha > 1$ there exist $s > s_0$ and a homotopy 
$$\psi : \E^u(\iota_s, c; \IR^n_s, \IR^n) \times [0,1] \lra \E^u(\iota_s, s; \IR^n_s, \IR^n)$$ 
such that {\rm (1)} for each $h \in \E^u(\iota_s, c; \IR^n_s, \IR^n)$ 
\begin{itemize}
\item[] \hspace{8mm} {\rm (i)} $\psi_0(h) = h$, \ \ {\rm (ii)} $\psi_1(h) = \id$ on $\IR^n_{\beta s}$, \ \ 
{\rm (iii)} $\psi_t(h) = h$ on $\IR^n_s - \IR^n_{\alpha s}$ \ $(t \in [0,1])$,  

\item[(2)] $\psi_t(\iota_s) = \iota_s$ \ $(t \in [0,1])$
\item[(3)] $\psi(\E^u(\iota_s, c; \IR^n_s, \IR^n_r) \times [0,1]) \subset \E^u(\iota_s, s; \IR^n_s, \IR^n_r)$ for any $r < s$. 
\end{itemize} 
\end{lemma}

\begin{proof} 
We apply Lemma~\ref{lemma_local_deform_E-end} to $0 < 1 < 2 \alpha < 2\beta$ and $\e = 1$. This yields $\delta \in (0, c/s_0)$ and a homotopy 
$$\phi : \E^u(\iota_1, \delta; \IR^n_1, \IR^n) \times [0,1] \lra \E^u(\iota_1, 1; \IR^n_1, \IR^n).$$ 
as in Lemma~\ref{lemma_local_deform_E-end}.  
Let $s := c/\delta$. Then $s > s_0$ and we have the homeomorphism 
$$\eta : \E^u(\IR^n_1, \IR^n) \approx \E^u(\IR^n_s, \IR^n) : \ \eta(f) =  k_s f \, k_{1/s}.$$
Since $\eta(\iota_1) = \iota_s$ and $\ds d(\eta(f), \eta(g)) = s \,d(f, g)$, for each $c > 0$ we have the restriction  
$$\eta_c : \E^u(\iota_1, a; \IR^n_1, \IR^n)) \approx \E^u(\iota_s, sa; \IR^n_s, \IR^n).$$  
Then the homotopy $\psi$ is defined by 
$$\psi_t = \eta_s \phi_t\eta_c^{-1}.$$ 
The conditions (1), (2) on $\psi$ follow from the corresponding properties of $\phi$. 
By (1)(iii) ${\rm Im}\, \psi_t(h) = {\rm Im}\, h$ for each $h \in \E^u(\iota_s, c; \IR^n_s, \IR^n)$, which implies (3). 
\end{proof} 

\subsection{Bi-Lipschitz Euclidean ends case} \mbox{} 

Suppose $(X, d)$ is a metric space and $L$ is a bi-Lipschitz $n$-dimensional Euclidean end of $X$. 
This means that $L$ is a closed subset of $X$ which admits a bi-Lipschitz homeomorphism 
$\theta : (\IR^n_1, \partial \IR^n_1) \cong ((L, {\rm Fr}_X L), d|_L)$ and 
$d(X - L, \theta(\IR^n_r)) \to \infty$ as $r \to \infty$. 
%A closed subset $L$ of $X$ is called a bi-Lipschitz $n$-dimensional Euclidean end of $(X, d)$ if 
%there exists a bi-Lipschitz homeomorphism 
%$\theta : ((L, {\rm Fr}_X L), d|_L) \cong (\IR^n_1, \partial \IR^n_1)$ and 
%$d(X - L, \theta(\IR^n_r)) \to \infty$ as $r \to \infty$. 
Let $\kappa \geq 1$ be the bi-Lipschitz constant of $\theta$ and 
for $a \geq 1$ let $L_a = \theta(\IR^n_a)$ and $\theta_a = \theta|_{\IR^n_a} : \IR^n_a \approx L_a$. 

\begin{lemma}\label{lemma_deform_homeo_1} 
For any $\lambda > 0$ and $s_0 \geq 1$ there exist $s > s_0$, $\mu > 0$ and 
a homotopy $\phi : {\cal H}^u(X; \lambda) \times [0,1] \lra {\cal H}^u(X; \mu)$
such that for each $h \in {\cal H}^u(X; \lambda)$
\begin{itemize}
\item[(i)\,] $\phi_0(h) = h$, \ \ {\rm (ii)} $\phi_1(h) = \id$ on $L_{3s}$, \ \ {\rm (iii)} $\phi_t(h) = h$ on $X - L_{2s}$ \ $(t \in [0,1])$,  
\item[(iv)] if $h = \id$ on $L_s$, then $\phi_t(h) = h$ $(t \in [0,1])$. 
\end{itemize} 
\end{lemma}

\begin{proof} Take any $\lambda > 0$. 
Since $d(X - L, L_r) \to \infty$ ($r \to \infty$), there exists 
\begin{itemize}
\item[(1)] $r > s_0$ such that $h(L_r) \subset L_1$ for any $h \in {\cal H}^u(X; \lambda)$. 
\end{itemize} 
Let $c \equiv \lambda \kappa > 0$. 
Applying Lemma~\ref{lemma_deformation} to $c, r$ and $\alpha=2$, $\beta=3$, 
we obtain $s > r$ and a homotopy 
$$\psi : \E^u(\iota_s, c; \IR^n_s, \IR^n_1) \times [0,1] \lra \E^u(\iota_s, s; \IR^n_s, \IR^n_1)$$ 
such that (2) for each $f \in \E^u(\iota_s, c; \IR^n_s, \IR^n_1)$ 
\begin{itemize}
\item[] \hspace{8mm} (i) $\psi_0(f) = f$, \ \ (ii) $\psi_1(f) = \id$ on $\IR^n_{3s}$, \ \ 
(iii) $\psi_t(f) = h$ on $\IR^n_s - \IR^n_{2s}$ \ $(t \in [0,1])$,  
\item[(3)] $\psi_t(\iota_s) = \iota_s$ \ $(t \in [0,1])$. 
\end{itemize} 

Consider the homeomorphism 
$$\Theta_s : \E^u(L_s, L_1) \approx \E^u(\IR^n_s, \IR^n_1) : \ \ \Theta_s(f) = \theta_1^{-1} f \hspace{0.5mm}\theta_s.$$ 
Since $\theta$ is $\kappa$-bi-Lipschitz, it is seen that $\Theta_s$ is also $\kappa$-bi-Lipschitz with respect to the sup-metrics. 
Since $\Theta_s(\iota_s^L) = \iota_s$, 
the maps $\Theta_s$ and $\Theta_s^{-1}$ restrict to   
$$\Theta_s : \E^u(\iota_s^L, \lambda; L_s, L_1) \lra \E^u(\iota_s, c; \IR^n_s, \IR^n_1) \hspace{4mm} \text{and} \hspace{4mm}   
\Theta_s^{-1} : \E^u(\iota_s, c; \IR^n_s, \IR^n_1) \lra \E^u(\iota_s^L, \kappa c; L_s, L_1).$$ 
Hence we obtain the homotopy 
$$\chi : \E^u(\iota_s^L, \lambda; L_s, L_1) \times [0,1] \lra \E^u(\iota_s^L, \kappa c; L_s, L_1) : \ \ \chi_t = (\Theta_s)^{-1} \psi_t \Theta_s.$$ 
From (2), (3) it follows that 
\begin{itemize}
\item[(4)] for each $h \in \E^u(\iota_s^L, \lambda; L_s, L_1)$ \\  
\hspace*{5mm} (i) $\chi_0(f) = f$, \ \ (ii) $\chi_1(f) = \id$ on $L_{3s}$, \ \ 
(iii) $\chi_t(f) = f$ on $L_s - L_{2s}$ \ $(t \in [0,1])$,  
\item[(5)] $\chi_t(\iota_s^L) = \iota_s^L$ \ $(t \in [0,1])$. 
\end{itemize} 

Since $s > r$, by (1) we have the restriction map 
$$R_s : {\cal H}^u(X; \lambda) \lra  \E^u(\iota_s^L, \lambda; L_s, L_1) : R_s(h) = h|_{L_s}.$$ 
Let $\mu =\kappa c$. 
Due to (4)(iii), the required homotopy is defined by 
$$\phi : {\cal H}^u(X; \lambda) \times [0,1] \lra {\cal H}^u(X; \mu) \ \ \ \text{by} \ \ \ 
\phi_t(h) = 
\left\{ 
\begin{array}[c]{@{\,}cl}
\chi_tR_s(h) & \text{on \ $L_s$} \\[2mm] 
h & \text{on \ $X - L_{2s}.$}
\end{array} \right.$$
\vskip -5mm 
\end{proof} 

\begin{lemma}\label{lemma_deform_homeo_2} 
For any $\lambda > 0$ and $r > r_0 \geq 1$ there exist $\lambda' > 0$ and 
a homotopy $\chi : {\cal H}^u(X; \lambda) \times [0,1] \lra {\cal H}^u(X; \lambda')$
such that for each $h \in {\cal H}^u(X; \lambda)$
\begin{itemize}
\item[(i)\,] $\chi_0(h) = h$, \ \ {\rm (ii)} $\chi_1(h) = \id$ on $L_r$, \ \ {\rm (iii)} $\chi_t(h) = h$ on $h^{-1}(X - L_{r_0}) - L_{r_0}$ $(t \in [0,1])$, 
\item[(iv)] if $h = \id$ on $L_{r_0}$, then $\chi_t(h) = h$ $(t \in [0,1])$. 
\end{itemize} 
\end{lemma}

\begin{proof} Let $s, \mu > 0$ and $\phi$ be as in Lemma~\ref{lemma_deform_homeo_1} with respect to $\lambda$ and $s_0 = r$. 
Using the product structure of $L$, we can find an isotopy  
$\xi : X \times [0,1] \to X$ such that 
\begin{itemize}
\item[] (a) $\xi_0 = \id_X$, \ \ (b) $\xi_1(L_r) = L_{3s}$, \ \ (c) $\xi_t = \id$ on $(X - L_{r_0}) \cup L_{4s}$ \ $(t \in [0,1])$. 
\end{itemize}
By (c) the map $[0,1] \ni t \lmt \xi_t \in {\cal H}^u(X)$ is continuous and $\nu \equiv \max \{ d(\xi_t, \id_X) \mid t \in [0,1] \} < \infty$. 
Thus, we obtain the homotopy 
$$\chi : {\cal H}^u(X; \lambda) \times [0,1] \lra {\cal H}^u(X) : \ \ \ \chi_t(h) = \xi_t^{-1} \phi_t(h) \xi_t.$$ 
Since $d(\xi_t^{-1}, \id_X) = d(\xi_t, \id_X) \leq \nu$, it follows that 
$d(\chi_t(h), \id_X) \leq \lambda' \equiv \mu + 2 \nu$ $(h \in {\cal H}^u(X; \lambda))$ and that ${\rm Im}\, \chi \subset {\cal H}^u(X; \lambda')$. 
The required conditions on $\chi$ follow from the properties of $\phi$ and $\xi$. 
\end{proof}

\begin{lemma}\label{lemma_deform_homeo_3} For any $r \in (1,2)$ there exists a homotopy $\psi : {\cal H}^u(X)_b \times [0,1] \lra {\cal H}^u(X)_b$
such that for each $h \in {\cal H}^u(X)_b$
\begin{itemize}
\item[(i)\,] $\psi_0(h) = h$, \ \ {\rm (ii)} $\psi_1(h) = \id$ on $L_2$, \ \ {\rm (iii)} $\psi_t(h) = h$ on $h^{-1}(X - L_r) - L_r$ $(t \in [0,1])$, 
\item[(iv)] if $h = \id$ on $L_r$, then $\psi_t(h) = h$ $(t \in [0,1])$, 
\item[(v)\,] for any $\lambda > 0$ there exists $\mu > 0$ such that $\psi_t({\cal H}^u(X; \lambda)) \subset {\cal H}^u(X; \mu)$ $(t \in [0,1])$. 
\end{itemize} 
\end{lemma}

\begin{proof} For $\lambda \geq 0$ let 
${\cal H}^u(X; \geq \hspace{-0.5mm}\lambda) = \{ h \in {\cal H}^u(X)_b \mid d(h, \id_X) \geq \lambda \}$.  
Take any sequence $r = r_1 < r_2 < \cdots < 2$. 
By repeated applications of Lemma~\ref{lemma_deform_homeo_2} we can find $\lambda_i > 0$ $(i \in \IN)$ and homotopies  
$$\chi^i : {\cal H}^u(X; \lambda_i+1) \times [0,1] \lra {\cal H}^u(X; \lambda_{i+1}) \ \ (i \in \IN)$$ 
such that for each $i \in \IN$ 
\begin{itemize}
\item[(1)] $\lambda_{i+1} > \lambda_i + 1$, 
\item[(2)] for each $h \in {\cal H}^u(X; \lambda_i+1)$ \\
\begin{tabular}[t]{c@{\ }l} 
(i) & $(\chi^i)_0(h) = h$, \hspace{3mm} (ii) $(\chi^i)_1(h) = \id$ on $L_{r_{i+1}}$, \\[2mm]
(iii) & $(\chi^i)_t(h) = h$ on $h^{-1}(X - L_{r_i}) - L_{r_i}$ $(t \in [0,1])$, \\[2mm]  
(iv) & if $h = \id$ on $L_{r_i}$, then $(\chi^i)_t(h) = h$ $(t \in [0,1])$. 
\end{tabular}
\end{itemize} 
\vskip 1mm 

For each $i \in \IN$ take a map 
\begin{itemize}
\item[(3)] $\alpha_i : {\cal H}^u(X; \lambda_i+1) \to [0,1]$ such that \ $\alpha_i(h) = 1$ if $d(h, \id_X) \leq \lambda_i$
\ and \ $\alpha_i(h) = 0$ if $d(h, \id_X) = \lambda_i+1$. 
\end{itemize} 
We modify $\chi^i$ to obtain the homotopy 
$$\eta^i : {\cal H}^u(X)_b \times [0,1] \lra {\cal H}^u(X)_b, \ \ 
(\eta^i)_t(h) = \left\{ \begin{array}[c]{@{\ }ll}
(\chi^i)_{\alpha_i(h) t}(h) & (h \in {\cal H}^u(X; \lambda_i+1)), \\[2mm] 
\ h & (h \in {\cal H}^u(X; \geq \hspace{-0.5mm}\lambda_i+1)). 
\end{array} \right.$$
Then, $\eta^i$ has the following properties: \ 
\begin{itemize}
\item[(4)] for each $h \in {\cal H}^u(X)_b$ \ \ 
(i) $(\eta^i)_0(h) = h$, \ \  (ii) $(\eta^i)_t(h) = h$ on $h^{-1}(X - L_{r_i}) - L_{r_i}$ $(t \in [0,1])$. 
\item[(5)] 
\begin{tabular}[t]{c@{\ }l}
(i) & $(\eta^i)_t(h) = h$ $(t \in [0,1])$ for any $h \in {\cal H}^u_{L_{r_i}}(X)_b \cup {\cal H}^u(X; \geq \hspace{-0.5mm}\lambda_i+1)$. \\[1.5mm]  
(ii) & $(\eta^i)_t({\cal H}^u(X; \lambda_i+1)) \subset {\cal H}^u(X; \lambda_{i+1})$ $(t \in [0,1])$. 
\end{tabular}
\vskip 1.5mm 
\item[(6)] $(\eta^i)_1({\cal H}^u(X; \lambda_i)) \subset {\cal H}^u_{L_{r_{i+1}}}(X)_b$. 
\end{itemize}
From (5) it follows that 
\begin{itemize}
\item[(7)] 
\begin{tabular}[t]{c@{\ }l}
(i) & $(\eta^j)_t({\cal H}^u(X;\lambda_i)) \subset {\cal H}^u(X;\lambda_i)$ \ \ $(j \leq i-1$, $t \in [0,1]$), \\[1.5mm]
(ii) & $(\eta^j)_t(h) = h$ $(h \in {\cal H}^u_{L_{r_{i+1}}}(X)_b)$ \ \ $(j \geq i+1$, $t \in [0,1]$).
\end{tabular} 
\end{itemize}
\vskip 1.5mm 
Hence we have 
\begin{itemize}
\item[(8)] 
\begin{tabular}[t]{c@{\ }l}
(i) & $(\eta^i)_1(\eta^{i-1})_1 \dots (\eta^1)_1({\cal H}^u(X;\lambda_i)) \subset (\eta^i)_1({\cal H}^u(X; \lambda_i)) \subset {\cal H}^u_{L_{r_{i+1}}}(X)_b$, \\[1.5mm]
(ii) & $(\eta^j)_t (\eta^{j-1})_1 \cdots (\eta^i)_1 \dots (\eta^1)_1(h) = (\eta^i)_1 \dots (\eta^1)_1(h)$ \ \ ($h \in {\cal H}^u(X;\lambda_i)$, $j \geq i+1$, $t \in [0,1]$). 
\end{tabular}  
\end{itemize}
\vskip 1mm 

Replacing $[0,1]$ by $[0, \infty]$, the homotopy $\psi : {\cal H}^u(X)_b \times [0,\infty] \lra {\cal H}^u(X)_b$ is defined by 
\vspace{1mm} 
$$\psi_t(h) = \left\{ \begin{array}[c]{@{\ }ll}
(\eta^j)_{t -j+1} (\eta^{j-1})_1 \cdots (\eta^1)_1(h) & (t \in [j-1, j], j \in \IN) \\[2mm] 
\ds \lim_{j \to \infty}(\eta^j)_1 \cdots (\eta^1)_1(h) & (t = \infty). 
\end{array} \right.$$
\vskip 1mm 
\noindent By (8)(ii) we have 
\begin{itemize}
\item[(9)] $\psi_t(h) = (\eta^i)_1 \dots (\eta^1)_1(h)$ \ \ ($h \in {\cal H}^u(X;\lambda_i)$, $t \in [i, \infty]$). 
\end{itemize}
This means that $\psi$ is well-defined and continuous. The required conditions on $\psi$ follow from (4) $\sim$ (8). 
For (v) note that $\psi_t({\cal H}^u(X; \lambda_i)) \subset {\cal H}^u(X; \lambda_{i+1})$ $(i \in \IN, t \in [0,1])$. 
\end{proof}

\begin{proposition}\label{prop_deform_homeo} 
For any $1 < s < r < 2$ there exists a strong deformation retraction $\phi$ of ${\cal H}^u(X)_b$ onto ${\cal H}^u_{L_r}(X)_b$ such that 
$$\mbox{$\phi_t(h) = h$ \ on \ $h^{-1}(X - L_s) - L_s$ \ \ for any \ $(h,t) \in {\cal H}^u(X)_b \times [0,1]$.}$$  
\end{proposition}

\begin{proof}  
Let $\psi : {\cal H}^u(X)_b \times [0,1] \lra {\cal H}^u(X)_b$ be the homotopy given by Lemma~\ref{lemma_deform_homeo_3}. 
Then $\psi$ is a deformation of ${\cal H}^u(X)_b$ into ${\cal H}^u_{L_2}(X)_b$ which fixes ${\cal H}^u_{L_r}(X)_b$ pointwise and 
satisfies 
\begin{itemize}
\item[(1)] $\psi_t(h) = h$ on $h^{-1}(X - L_r) - L_r$ \ \ $(h \in {\cal H}^u(X)_b, t \in [0,1])$.  
\end{itemize} 

Let $Y = X - {\rm Int}\,L_3$, $S = L_3 - {\rm Int}\,L_3$ and $N = {\rm Int}\,L - {\rm Int}\,L_3$. 
Then $N$ is an open collar neighborhood of $S$ in $Y$ and 
for any $s \in (0, r)$ it admits a parametrization 
\begin{itemize}
\item[(2)] $\vartheta : (S \times [0,4), S \times \{ 0 \}) \approx (N, S)$ \ \ such that \ \  
$N_1 = L_2 - {\rm Int}\,L_3$, \ 
$N_2 = L_r - {\rm Int}\,L_3$, \ 
$N_3 = L_s - {\rm Int}\,L_3$. 
\end{itemize} 
Here, $N_s = \theta(S \times [0,s])$ $(s \in [0,4))$. 
Under the canonical identification $({\cal H}^u_{L_2}(X)_b, {\cal H}^u_{L_r}(X)_b) \approx ({\cal H}^u_{N_1}(Y)_b, {\cal H}^u_{N_2}(Y)_b)$, 
Lemma~\ref{lemma_collar} yields a strong deformation retraction 
$\chi_t$ $(t \in [0,1])$ of ${\cal H}^u_{L_2}(X)_b$ onto ${\cal H}^u_{L_r}(X)_b$ such that 
\begin{itemize}
\item[(3)] $\chi_t(h) = h$ \ on \ $h^{-1}(X - L_s) - L_s$ \ \ for any \ $(h,t ) \in {\cal H}^u_{L_2}(X)_b \times [0,1]$.
\end{itemize} 

Finally, the homotopy $$\phi : {\cal H}^u(X)_b \times [0,1] \lra {\cal H}^u(X)_b : \hspace{3mm} 
\phi_t = 
\left\{\begin{array}[c]{ll}
\psi_{2t} & (t \in [0,1/2]), \\[2mm] 
\chi_{2t-1} \psi_1 & (t \in [1/2,1]) 
\end{array}\right.$$ 
\vskip 2mm 
\noindent is a strong deformation retraction of ${\cal H}^u(X)_b$ onto ${\cal H}^u_{L_r}(X)_b$ satisfying the required condition. 
\end{proof}

\begin{proof}[\bf Proof of Theorem~\ref{thm_Euclid-end}] 
For each $i \in [m]_+$ we can replace the bi-Lipschitz homeomorphism $\theta_i$ for $L_i$ by another $\theta_i'$ such that $L_i' = \theta_i'(\IR^{n_i}_{4/3})$ and $L_i'' = \theta_i'(\IR^{n_i}_{3/2})$. 
Then, by Proposition~\ref{prop_deform_homeo} 
there exists a strong deformation retraction $\phi^i$ of ${\cal H}^u(X)_b$ onto ${\cal H}^u_{L_i''}(X)_b$ such that 
$$\mbox{$(\phi^i)_t(h) = h$ \ on \ $h^{-1}(X - L_i') - L_i'$ \ \ for any \ $(h,t) \in {\cal H}^u(X)_b \times [0,1]$.}$$  

Define the homotopy $\phi$ by $\phi_t = (\phi^m)_t \cdots (\phi^1)_t$ $(t \in [0,1])$. 
\end{proof}

%%%%%%%%%%%%%%%%%

%%%%%%%%%%%%%
\end{document}